\documentclass[intlim,righttag,10pt]{amsart}
\usepackage{enumitem}
\usepackage{amscd}
\usepackage{amssymb}
\usepackage{graphicx} 
\usepackage[all]{xy}
\usepackage{fge}

\usepackage{color}
\usepackage[textsize=tiny]{todonotes}

\RequirePackage[T1]{fontenc}
\RequirePackage{amsfonts,latexsym,amssymb}

\RequirePackage{mathrsfs}
\let\mathcal\mathscr

\let\bb\mathbb

\usepackage{bbm}
\usepackage{hyperref}

\newtheorem{theorem}[equation]{Theorem}
 \newtheorem{lemma}[equation]{Lemma}
 \newtheorem{proposition}[equation]{Proposition}
 \newtheorem{corollary}[equation]{Corollary}

\theoremstyle{definition}

\newtheorem{remark}[equation]{Remark}

\theoremstyle{remark}
\newtheorem*{acknowledgments}{Acknowledgments}

\def\invlim{\mathop{\vtop{\ialign{##\crcr$\hfill{\lim}\hfil$\crcr
\noalign{\kern1pt\nointerlineskip}\leftarrowfill\crcr\noalign
{\kern -3pt}}}}\limits}
\def\dirlim{\mathop{\vtop{\ialign{##\crcr$\hfill{\lim}\hfil$\crcr
\noalign{\kern1pt\nointerlineskip}\rightarrowfill\crcr\noalign
{\kern -3pt}}}}\limits} 
\def\lomapr#1{\smash{\mathop{\relbar\joinrel\longrightarrow}\limits^{#1}}}

\def\phi{\varphi}
\def\epsilon{\varepsilon}

\newcommand{\Spa}{\operatorname{Spa}}

\newcommand{\ovk}{\overline{K} }

\newcommand{\dr}{\operatorname{dR} } 
\newcommand{\hk}{\operatorname{HK} }

\newcommand{\gr}{\operatorname{gr} }

\newcommand{\colim}{\operatorname{colim} }

\newcommand{\proeet}{\operatorname{pro\acute{e}t} }
\newcommand{\qproeet}{\operatorname{qpro\acute{e}t} }
\newcommand{\eet}{\operatorname{\acute{e}t} }

\newcommand{\Gal}{\operatorname{Gal} }
\newcommand{\can}{ \operatorname{can} }

\newcommand{\st}{\operatorname{st} }
\newcommand{\kker}{\operatorname{Ker} } 
 
\newcommand{\coker}{\operatorname{coker} }

\newcommand{\sg}{{\mathcal{G}}}

\newcommand{\so}{{\mathcal O}}

\newcommand{\sd}{{\mathcal{D}}}

\newcommand{\wt}{\widetilde}
\newcommand{\wh}{\widehat}
\numberwithin{equation}{section}

\def\R{{\mathrm R}}

  \def\B{{\bf B}}
\def\Q{{\bf Q}} \def\Z{{\bf Z}}

\def\N{{\bf N}}

\def\bdr{{\bf B}_{{\rm dR}}}
 
\def\Bdr{{\mathbb B}_{{\rm dR}}}

\def\rg{{\rm R}\Gamma}

\def\epsilon{\varepsilon}

\def\FS{{\mathscr{F}}}

\numberwithin{equation}{section}
\begin{document}
\title[On the $v$-Picard group of Stein spaces]{On the $v$-Picard group of Stein spaces}
\author{Veronika Ertl} 
\address{Instytut Matematyczny PAN, ul. \'Sniadeckich 8, 00-656 Warszawa, Poland }
\email{vertl@impan.pl} 
\author{Sally Gilles}
\address{Institute for Advanced Study, 1 Einstein Drive,  Princeton, NJ 08540, USA }
\email{ gilles@ias.edu}
\author{Wies{\l}awa Nizio{\l}}
\address{CNRS, IMJ-PRG, Sorbonne Universit\'e, 4 place Jussieu, 75005 Paris, France}
\email{wieslawa.niziol@imj-prg.fr}
\date{\today}
\thanks{W.N.'s research was supported in part by the grant NR-19-CE40-0015-02 COLOSS. S.G.'s research was partially supported by the National Science Foundation under Grant No. DMS-1926686. V.E.'s research was supported in part by by the DFG grant SFB 1085 ``Higher Invariants''.  }
\maketitle

 \begin{abstract}
We study the image of the Hodge-Tate logarithm map (in any cohomological degree), defined by Heuer, in the case of smooth Stein varieties. 
Heuer, motivated by the computations for the affine space of any dimension, raised the question whether this  image is always equal to the group of closed differential forms. We show that it indeed always contains such forms  but the quotient  can be non-trivial:
 it contains a slightly mysterious $\Z_p$-module that  maps,  via the Bloch-Kato exponential map, to integral classes in the pro-\'etale cohomology. This quotient is already non-trivial for open unit discs of dimension strictly greater than $1$.
 \end{abstract}
\tableofcontents

\section{Introduction}
Let $\so_K$ be a complete discrete valuation ring with fraction field $K$  of characteristic 0 and with perfect
residue field $k$ of positive characteristic $p$.
Let $C$ be the $p$-adic completion of an algebraic closure of $K$. 

In \cite{Heu21}, Heuer constructed a  {\em Hodge-Tate logarithm} map ${\rm HTlog}$  such that, for any smooth rigid analytic space $X$ over $C$, there is an exact sequence
\begin{equation}
\label{heuer}
 0 \to {\rm Pic}_{\rm an}(X) \to {\rm Pic}_v(X^{\diamond}) \xrightarrow{\rm HTlog} \Omega^1(X)(-1),
\end{equation}
where $X^{\diamond}$ denotes the associated diamond and the $(-1)$ appearing in the last term is the Tate twist. 
Moreover, he proved the following result:
\begin{theorem} {\rm (Heuer, \cite[Th.\,1.3, Th.\,6.1]{Heu21})} \label{morn1}
Let $X$ be a smooth rigid analytic space over $C$. 
\begin{enumerate}
\item If $X$ is proper or a curve, then the map ${\rm HTlog}$ from~\eqref{heuer} is surjective. 
\item If $X$ is the affine space ${\mathbb A}^d_C$ of dimension $d$, then the image of ${\rm HTlog}$ is equal to the kernel of the differential, i.e., we have an exact sequence: 
\begin{equation}\label{batak1}
 0 \to {\rm Pic}_{\rm an}(X) \to {\rm Pic}_v(X^{\diamond}) \xrightarrow{\rm HTlog} \Omega^1(X)^{d=0}(-1) \to 0.  
\end{equation}
\end{enumerate}
\end{theorem}
Note that, in the case of the affine space, the analytic Picard group is trivial and the exact sequence~\eqref{batak1} gives in fact an isomorphism: 
\[ {\rm Pic}_v(X^{\diamond}) \xrightarrow[\sim]{\rm HTlog} \Omega^1(X)^{d=0}(-1). \] 

Heuer has also raised the question (see \cite[Rem.\,6.10]{Heu21}) whether we have an analogue of the exact sequence \eqref{batak1} for any smooth Stein space over $C$, i.e., whether the image of ${\rm HTlog}$ is equal to the closed differential forms. 
Using  a  simple functoriality argument he has shown that, for all smooth rigid analytic spaces,  this image contains all exact forms.

The goal of this paper is to extend Theorem \ref{morn1} and to compute the image of the Hodge-Tate logarithm for more general Stein rigid analytic spaces. More precisely, we show the following: 

\begin{theorem}
\label{main}
Let $X$ be a smooth Stein rigid analytic space over $C$ and let $i\geq 1$. Then, the image of the restriction of the Hodge-Tate logarithm to the cohomology group of principal units 
\[ {\rm HTlog}_U : H^i_v(X^{\diamond},U) \to \Omega^i_{X}(X)(-i) \]
fits into a short exact sequence of abelian groups
\begin{equation}
\label{image0}
0 \to  \Omega^i(X)(-i)^{d=0} \to {\rm Im}({\rm HTlog}_U) \xrightarrow{ {\rm Exp}}  \mathcal{I}^i(X) \to 0 
\end{equation}
where the $\Z_p$-module $\mathcal{I}^i(X) \subset H^{i+1}_{\proeet}(X, \Q_p(1))$ is the intersection ${\rm Im}({\rm Exp}) \cap {\rm Im}(\iota)$, where
$\iota: H^{i+1}_{\proeet}(X, \Z_p(1)) \to H^{i+1}_{\proeet}(X, \Q_p(1))$ is the canonical map. 
\end{theorem}

The map ${\rm Exp}$ in \eqref{image0}  is the {\em Bloch-Kato exponential} map 
\begin{equation}
\label{CDN}
 {\rm Exp} : (\Omega^i(X)/{\rm Ker}\,d)(-i) \hookrightarrow H^{i+1}_{\proeet}(X, \Q_p(1)) \end{equation}
from  the fundamental diagram of Colmez-Dospinescu-Nizio{\l} computing the $p$-adic pro-\'etale cohomology of Stein spaces \cite{CDN3}, \cite{CN4}. 

In particular, the above result shows that the closed differential forms are all in the image of the  morphism ${\rm HTlog}$. In the cases where the group $\mathcal{I}^i(X)$ is non-trivial, it also shows that this  image is larger than what the result for the affine space was suggesting, since it also includes the differential forms coming from the integral pro-\'etale cohomology. We show that 
group $\mathcal{I}^i(X)$ is non-trivial already in  the case of  the unit open disc of dimension at least $2$. 

  For Stein spaces which in addition  are almost proper varieties or analytifications of algebraic varieties we show that the group $\mathcal{I}^i(X)$ is trivial, for all $i$.  Moreover, if  the image of ${\rm HTlog}_U$ coincides with the one of ${\rm HTlog}$, we obtain the  exact sequence:
  $$
   0 \to {\rm Pic}_{\rm an}(X) \to {\rm Pic}_v(X^{\diamond}) \xrightarrow{\rm HTlog} \Omega^1(X)^{d=0}(-1) \to 0.  
$$  This suggests that $v$-line bundles are a twisted form of closed Higgs line bundles in this case.  Such results are  related to the $p$-adic Simpson correspondence studied by Faltings in~\cite{Fal05} and Deninger-Werner in~\cite{DW05}. In~\cite{Heu21}, Heuer used the first point of Theorem~\ref{morn1} to construct a $p$-adic Simpson correspondence in rank~1 (see Theorem~1.7 in {\em loc.\,cit.}). He later extended his result to higher ranks in~\cite{Heu23}. Our result suggests that a similar correspondence could be true for Stein almost proper varieties or Stein analytifications of algebraic varieties. 

\vskip.2cm
 
To prove Theorem~\ref{main}, the strategy we follow is similar to the one of Heuer in the case of the affine space (see \cite[Sec.\,6.2]{Heu21}): we compare the image of the Hodge-Tate logarithm to the kernel of the map 
from $\Omega^i(X)(-i)$ to the pro-\'etale cohomology. In the work of Heuer, this map is defined as the boundary morphism 
coming from the fundamental exact sequence of $p$-adic Hodge Theory: 
\[ 0 \to \Q_p(1) \to \mathbb{B}^{\varphi=p} \to  \mathbb{B}^+_{\rm dR}/t \simeq \so \to 0\] 
In the case of the affine space, Heuer was able to compute this kernel using the computation of the cohomologies of $\mathbb{B}^{\varphi=p}$ and  $\mathbb{B}_{\dr}$ by Le Bras in \cite{LB18} and the Poincar\'e Lemma of Scholze  \cite{Sch13a}. However, for a general Stein space, this is not sufficient: to determine completely the kernel we need to use the Bloch-Kato exponential~\eqref{CDN} from Colmez-Dospinescu-Nizio{\l} (in the form defined by Bosco in \cite{Bos23}); the fact that this map is injective follows from  slope properties of   Hyodo-Kato cohomology of $X$. The equality between the map ${\rm Exp}$ and the one used by Heuer is checked in Section~\ref{compatibility}.  

We present in Section~\ref{examples} a number of examples: a torus,  a Drinfeld space, analytifications of algebraic varieties, an open disc, where we try to determine the $v$-Picard group. Having Theorem~\ref{main}, the main difficulty is in proving that the Picard group can be computed using the sheaf of principal units. We end the paper with a discussion of  the groups $H_{\tau}^i(X, {\bb G}_m)$, $i\ge 2$ and $\tau \in \{ \eet, \proeet, v\}$, where $X$ is a rigid analytic Stein space: similar tools as those used  in the study of the Picard group allow us to partially compute such groups. In particular, in the case of the affine space ${\bb A}_C^n$, we obtain that, for $i \ge 2$,  
\[ H_{\eet}^i({\bb A}_C^n,  {\bb G}_m) \simeq \Omega^{i}({\bb A}_C^n)^{d=0}(-i+1) \quad \text{and} \quad H_{\proeet}^i({\bb A}_C^n,  {\bb G}_m) \simeq H_v^i({\bb A}_C^n,  {\bb G}_m) \simeq \Omega^{i}({\bb A}_C^n)^{d=0}(-i),\]
and the canonical map $H_{\eet}^i({\bb A}_C^n,  {\bb G}_m)\to H_{\proeet}^i({\bb A}_C^n,  {\bb G}_m)$ is the zero map. This gives a non-trivial example of a rigid analytic variety for   which one can compute explicitly the Brauer group.
 
\begin{acknowledgments}This paper is a product of a collaboration started during WINE4 (Women in Numbers Europe -- 4), in the summer of 2022 at Utrecht. We would like to thank the organisers for giving us the opportunity to work together in such a great place !  Special thanks go to Annette Werner who co-organised with W.N.\ the research group  on ``Relative $p$-adic Hodge Theory'' as well as to members of the group for many discussions on  $p$-adic geometry. We would also like to thank Piotr Achinger, Pierre Colmez, Gabriel Dospinescu, Ben Heuer, Damien Junger and Peter Scholze for helpful comments  concerning the content of this paper. Special thanks go to the referee for a very careful reading of the original manuscript and for many helpful comments and suggestions that have substantiailly improved the paper. 

 S.G.\ would like to thank the MPIM of Bonn and the IAS of Princeton for their support and hospitality during the academic years 2022-2023 and 2023-2024 when parts of this paper were written.  
V.E.\ would like to thank the IMPAN in Warsaw for its support and hospitality during the academic year 2023-2024. 
\end{acknowledgments}

\subsubsection*{Notation and conventions.}\label{Notation}
Let $\so_K$ be a complete discrete valuation ring with fraction field $K$  of characteristic 0 and with perfect
residue field $k$ of characteristic $p$. Let $\ovk$ be an algebraic closure of $K$ and let $\so_{\ovk}$ denote the integral closure of $\so_K$ in $\ovk$. Let $C=\wh{\ovk}$ be the $p$-adic completion of $\ovk$.  Let
$W(k)$ be the ring of Witt vectors of $k$ with 
 fraction field $F$ (i.e., $W(k)=\so_F$). We denote by $\breve{F}$ the completion of the maximal unramified extension of $K$. Set $\sg_K=\Gal(\overline {K}/K)$ and 
let $\phi$ be the absolute
Frobenius on $W(\overline {k})$. 
We will denote by $ \wh{\B}_{\st},\B_{\dr}$ the  semistable and  de Rham period rings of Fontaine. 

All rigid analytic spaces considered will be over $K$ or $C$. 
We assume that they are separated, taut\footnote{That is,   for  all quasi-compact  opens $V$ of  $X$,  the closure $\overline{V}$ of  $V$ in  $X$ is quasi-compact.}, and countable at infinity. We say that a rigid analytic space $X$ is Stein if it has an admissible affinoid covering $X = \bigcup_{i \in \N}U_i$ such that $U_i \Subset U_{i+1}$,
i.e., the inclusion $U_i \subset U_{i+1}$ factors through the adic compactification of $U_i$. Recall that, for such spaces, coherent sheaves are acyclic.

\section{Preliminaries}
\label{prelim}

We gather here the basic facts needed later on in the paper. 

\subsection{Vector bundles in the $v$-topology}
We gather here a few facts about $v$-vector bundles. 

Recall that the $v$-topology on a perfectoid space $X$ is defined as the topology whose covers are generated by all open covers (in the analytic topology) and by all the surjective maps of affinoids (see \cite[Lecture~17]{SW20}). We have that all diamonds are $v$-sheaves.  
If $Y$ is a diamond over ${\rm Spd}(C)$, a $v$-sheaf $V$ is a $v$-vector bundle of rank $n$, $n\in \N$, on $Y$ if it is a ${\rm GL}_n^{\diamond}$-torsor for the $v$-topology (where ${\rm GL}_n^{\diamond}$ denotes the diamond associated to the usual rigid space ${\rm GL}_n$). If $q: X \to Y$ is a $v$-cover of diamonds and $V$ a vector bundle on $X$ then every descent datum on $V$ is effective, i.e. the descent datum comes from a $v$-vector bundle on $Y$ (see \cite[Def.\,2.5 and Lem.\,2.6]{Heu21}). In particular, the $v$-vector bundles of rank $n$ on a diamond $Y$ (up to isomorphism) are classified by $H^1_v(Y, {\rm GL}^{\diamond}_n)$. In this paper, we are interested in  the group of line bundles: 
\[ {\rm Pic}_v(Y):= H^1_v(Y,{\rm GL}^{\diamond}_1). \]

 The diamond $Y$ can also be equipped with the \'etale topology and the quasi-pro-\'etale topology (see \cite[Sec.\,9.2]{SW20}). If $Y$ comes from a rigid space $X$ (i.e. $Y=X^{\diamond}$), then we have an equivalence $Y_{\eet}\simeq X_{\eet}$ (\cite[Th.\,10.4.2]{SW20}) and we have the following inclusion of sites: 
\[ X_{\rm an} \subset X_{\eet} \simeq X^{\diamond}_{\eet} \subset X_{\proeet} \subset X^{\diamond}_{\qproeet} \subset X^{\diamond}_v. \]
If $X$ is an affinoid perfectoid, we know from a result of Kedlaya-Liu \cite[Th.\,3.5.8]{KL16} that the notions of vector bundles in all these topologies coincide. The pro-\'etale, quasi-pro-\'etale and $v$-topologies being locally affinoid perfectoid, it follows that for a general smooth rigid space $X$, we also have that the groups of vector bundles in these three topologies are equal. It is also known (\cite[Prop.\,8.2.3]{FVdP}) that ${\rm Pic}_{\rm an}(X) \simeq {\rm Pic}_{\eet}(X)$. We are left to study the map 
\begin{equation}
\label{et-v}
{\rm Pic}_{\eet}(X) \hookrightarrow {\rm Pic}_{v}(X^{\diamond}). \end{equation}   
 
\subsection{Topologies on $X$}

Let $X$ be a smooth rigid space over $C$ and $X^{\diamond}$ be the associated diamond. In the following we write $X_v$ for the site $X_v^{\diamond}$. 
We denote by $\so_{\eet}$ and $\so_{v}$ the structure sheaves for the \'etale and the $v$-topology, and by $\so_{\proeet}$ the \textit{completed} pro-\'etale structure sheaf.
For $\tau$ one of these topologies, we also denote by  $\so^{\times}_{\tau}$ the sheaf of invertible functions, 
by $\so^{+}_{\tau}$  the sheaf of integral elements and 
by $U_{\tau}:=1+ \mathfrak{m}_{C}\so^+_{\tau} \subset \so_{\tau}^{\times}$  the sheaf of principal units. 
Let $\overline{\so}^{\times}$ be the quotient of $\so_{\tau}^{\times}$ by $U_{\tau}$. For $\mathcal{G} \in \{ \so_{\tau}^{\times}, \overline{\so}_{\tau}^{\times} \}$, we write $\mathcal{G}[\tfrac{1}{p}]$ for the sheaf of abelian groups $\varinjlim_{x \mapsto x^p} \mathcal{G}$. If $X$ is quasi-compact then 
 $\so_{\tau}^{\times}[\tfrac{1}{p}](X)= \so_{\tau}^{\times}(X)[\tfrac{1}{p}]$.  By \cite[Lem.\,2.16]{Heu21}, we have $\overline{\so}_{\tau}^{\times}[\tfrac{1}{p}]\simeq \overline{\so}_{\tau}^{\times}$. 
 
We summarise in the following proposition the various equalities that we have between the $H^1$-groups of these sheaves: 

\begin{proposition}
\label{topo} We denote by $\nu : X_v \to X_{\eet}$ the canonical morphism. It decomposes as $X_v \xrightarrow{\lambda} X_{\proeet} \xrightarrow{\mu} X_{\eet}$. 
\begin{enumerate}
\item The sheaf $\so^{\times}$: we can interchange pro-\'etale and $v$-topologies: 
\begin{align*}
H^1_{\proeet}(X, \so^{\times}) \stackrel{\sim}{\to} H^1_{v}(X, \so^{\times}).  
\end{align*}
\item The sheaf $\overline{\so}_v^{\times}$: we can interchange all three topologies: 
\begin{align*}
&\lambda_*\overline{\so}_v^{\times}\simeq \overline{\so}_{\proeet}^{\times}, \quad \R^1\lambda_*\overline{\so}^{\times}_v=0, \quad \text{ and in particular, } \quad  H^1_{\proeet}(X, \overline{\so}^{\times}) \stackrel{\sim}{\to} H^1_{v}(X, \overline{\so}^{\times})  \\
&\mu_*\overline{\so}_{\proeet}^{\times}\simeq \overline{\so}_{\eet}^{\times}, \quad \R^1\mu_*\overline{\so}^{\times}_{\proeet}=0, \quad \text{ and in particular, } \quad  H^1_{\eet}(X, \overline{\so}^{\times}) \stackrel{\sim}{\to} H^1_{\proeet}(X, \overline{\so}^{\times}) \\
&\nu_*\overline{\so}_{v}^{\times}\simeq \overline{\so}_{\eet}^{\times}, \quad \R^1\nu_*\overline{\so}^{\times}_v=0, \quad \text{ and in particular, } \quad H^1_{\eet}(X, \overline{\so}^{\times}) \stackrel{\sim}{\to} H^1_{v}(X, \overline{\so}^{\times}) .
\end{align*}
 \item More generally, we have natural isomorphisms 
  \begin{equation}\label{even2}
  \R\nu_*\overline{\so}^{\times}=\overline{\so}^{\times},\quad \R\mu_*\overline{\so}^{\times}=\overline{\so}^{\times}.
  \end{equation}
\end{enumerate}
\end{proposition}

\begin{proof}
The first point follows from the result of Kedlaya-Liu \cite[Th.\,3.5.8]{KL16}, as explained above. For the second point, see the proof of Lemma~2.22 in \cite{Heu21}.  The third claim is proved in  \cite[Th.\,1.7, Cor.\,2.11]{Heu2}.
\end{proof}

We recall now the exponential and logarithm  maps from \cite{Heu21}. 
They will be used to define the Hodge-Tate logarithm ${\rm HTlog}$. 
The logarithm exact sequence stated below will play an important role in the computation of the cokernel of the map~\eqref{et-v} as it will allow us to compare it to the $p$-adic pro-\'etale cohomology. 
   
The usual $p$-adic exponential and logarithm maps 
${\rm exp}(x)= \sum_n x^n/n!$ and 
{$\log(x)=\sum_{n\ge 1}\frac{(1-x)^n}{n}$} 
define morphisms of sheaves $${\rm exp}: p'\so^{+} \to 1+ p'\so^{+} \text{ and } {\rm log}:1+ \mathfrak{m} \so^{+} \to \so$$ where $p'=p$ if $p>2$ and $p'=4$ if $p=2$, such that $\log(1 +p'\so^{+})\subset p'\so^+$, ${\rm exp} \circ \log={\rm Id}$ on $1 +p'\so^+$ and $\log\circ\exp ={\rm Id}$ on {$p'\so^{+}$}. We have the following result:

 \begin{lemma}{\rm (Heuer, \cite[Lem.\,2.18, Prop.\,2.21]{Heu21})}
 \label{log}
 Let $X$ be a smooth rigid space over $C$ and $\nu: X_v \to X_{\eet}$ and $\mu: X_{\proeet} \to X_{\eet}$ as before. 
 \begin{enumerate}
 \item For $\tau\in\{v, {\eet}, \proeet\}$, there are exact sequences of sheaves on $X_{\tau}$: 
 \begin{align}\label{even1}
& 1\to (\Q_p/\Z_p)(1) \to U_{\tau} \xrightarrow{{\rm log}}  \so_{\tau} \to 1, \\ 
& 1\to \so_{\tau} \xrightarrow{{\rm exp}} \so_{\tau}^{\times}[\tfrac{1}{p}] \to {\overline{\so}}_{\tau}^{\times} \to 1. \notag
  \end{align} 
  \item Let $i\ge 1$. The  maps \eqref{even1} and the isomorphisms \eqref{even2}  induce natural isomorphisms 
 \begin{align*}
 &{\rm log}: \R^i\nu_* U \xrightarrow{\sim} \R^i\nu_*\so \quad \text{ and } \quad  \R^i\mu_* U \xrightarrow{\sim} \R^i\mu_*\so; \\ 
 &{\rm exp}: \R^i\nu_*\so \xrightarrow{\sim} \R^i\nu_*\so^{\times} \quad \text{ and } \quad  \R^i\mu_*\so \xrightarrow{\sim} \R^i\mu_*\so^{\times}. 
 \end{align*}
 \end{enumerate}
 \end{lemma}

\subsection{The Hodge-Tate logarithm}

We recall here the definition of the Hodge-Tate logarithm from \cite{Heu21}. 
\begin{proposition}[Hodge-Tate morphisms]
\label{omega}
 Let $X$ be a smooth rigid space over $C$, $\nu: X_v \to X_{\eet}$ and $\mu: X_{\proeet} \to X_{\eet}$ as before. Then for all $i \ge 0$, there are $\so_X$-linear isomorphisms: 
\[ {\rm HT}: \R^i\nu_*\so\xrightarrow{\sim}\Omega^i_X(-i),\quad {\rm HT}: \R^i\mu_*\so\xrightarrow{\sim}\Omega^i_X(-i)  \quad \text{on } X_{\eet}. \]
\end{proposition} 
For $\mu$, the result is due to Scholze in \cite[Cor.\,6.19]{Sch13a}, \cite[Prop.\,3.23]{Sch13b}: the Hodge-Tate morphism ${\rm HT}$ is defined as the inverse of the connecting morphism in the  Faltings extension (see Section~\ref{faltings-ext}). We obtain the result for $\nu$ using that $\R\lambda_*\so_v=\so_{\proeet}$.   

The Leray spectral sequence for the morphism $\nu : X_v \to X_{\eet}$ 
and the sheaf $\so^{\times}$
\begin{equation}\label{leray_nu}
E_2^{i,j}= H^i_{\eet}(X, \R^j\nu_{*} \so^{\times}) \Rightarrow H^{i+j}_{v}(X, \so^{\times}) 
\end{equation}
induces an exact sequence: 
\begin{equation}\label{even3} 0 \to H^1_{\eet}(X, \nu_{*} \so^{\times}) \to H^1_{v}(X,  \so^{\times}) \to H^0_{\eet}(X, \R^1\nu_{*} \so^{\times}) \to H^2_{\eet}(X, \nu_{*} \so^{\times}) \to H^2_v(X,  \so^{\times}).
\end{equation}
For $i\geq 1$, we define the Hodge-Tate logarithm
$$
{\rm HTlog}_i: H^i_v(X,  \so^{\times}) \to H^0_{\eet}(X, \Omega_X^i(-i))
$$
 as the composition:
\begin{equation}
\label{HTlog}
{\rm HTlog}_i: H^i_v(X,  \so^{\times}) \to H^0_{\eet}(X, \R^i\nu_{*} \so^{\times}) \xleftarrow[\sim]{{\rm exp}} H^0_{\eet}(X, \R^i\nu_{*} \so) \xrightarrow{\sim} H^0_{\eet}(X, \Omega_X^i(-i)),
\end{equation}
where the first arrow is the edge map of the Leray spectral sequence \eqref{leray_nu}, the second one is the exponential map from Lemma~\ref{log},  and the third one is the isomorphism from Proposition~\ref{omega}. We obtain an analogous morphism replacing the $v$-topology by the pro-\'etale one.

  Since $\nu_*\so^{\times}=\so^{\times}$ and using the map \eqref{HTlog}, we can rewrite the exact sequence \eqref{even3} as 
$$
 0 \to {\rm Pic}_{\rm an}(X) \to  {\rm Pic}_{v}(X)\lomapr{\rm HTlog} \Omega^1(X)(-1) \to H^2_{\eet}(X,  \so^{\times}) \to H^2_v(X,  \so^{\times}) 
$$
\begin{remark}
By Lemma \ref{log}, we can also consider the restriction of ${\rm HTlog}$ to $U$ and use the pro-\'etale and $v$-topology interchangeably.  
For $i\geq 1$, we define the Hodge-Tate logarithm
$$
{\rm HTlog}_{{U,i}}: H^i_v(X,  U) \to H^0_{\eet}(X, \Omega_X^i(-i))
$$
similarly as the composition:
\begin{equation}
\label{HTlogU}
{\rm HTlog}_{{U,i}}: H^i_v(X, U) \to H^0_{\eet}(X, \R^i\nu_{*} U) \xrightarrow[\sim]{{\rm log}} H^0_{\eet}(X, \R^i\nu_{*} \so) \xrightarrow{\sim} H^0_{\eet}(X, \Omega_X^i(-i)).
\end{equation}
It is compatible with the map ${\rm HTlog}$ from \eqref{HTlog}. 
\end{remark}

\begin{remark}
\label{HTstein}
For a Stein space, we know that for a coherent sheaf $\mathcal{F}$, the group $H^i_{\eet}(X, \mathcal{F})$ is zero for $i >0$. In particular, we obtain that $H^i_{\eet}(X, \R^j\mu_*\so)=0$, $i\ge 1$,  and hence the Leray spectral sequence for  $\mu: X_{\proeet} \to X_{\eet}$ and the sheaf $\so$,
induces an isomorphism 
$$ H^i_{\proeet} (X, \so) \xrightarrow{\sim} H^0_{\eet}(X, \R^i\mu_*\so).$$
In the following, when $X$ is Stein, we still write ${\rm HT}$ for the composition: 
\[ {\rm HT} : H^i_{\proeet} (X, \so) \xrightarrow{\sim} H^0_{\eet}(X, \R^i\mu_*\so) \xrightarrow{\sim} \Omega^i(X)(-i). \] 
And similarly for the $ v$-topology:
\[ {\rm HT} : H^i_{v} (X, \so) \xrightarrow{\sim} H^0_{\eet}(X, \R^i\nu_*\so) \xrightarrow{\sim} \Omega^i(X)(-i). \] 
\end{remark}

\subsection{The Leray spectral sequence for Stein spaces} 
We show here that,  in the case of smooth Stein spaces, the Leray spectral sequence for the projection $\nu: X_v\to X_{\eet}$ and the sheaf $\so^{\times}$ simplifies enormously. 

\begin{proposition}
\label{ses-leray}
Let $X$ be a smooth Stein rigid analytic variety of dimension $d$ over $C$. Then:
\begin{enumerate} 
\item $H^0_{\eet}(X, \so^{\times}) \simeq H^0_{v}(X, \so^{\times})$ and $H^i_{\eet}(X, \so^{\times}) \simeq H^i_{v}(X, \so^{\times})$ for $i \ge d+1$.  

\item We have exact sequences: 
\begin{align*}
&0 \to H^1_{\eet}(X, \so^{\times}) \to H^1_{v}(X, \so^{\times}) \xrightarrow{{\rm HTlog}_1} {\rm Ker}({\partial_2}:
\Omega^1(X)(-1) \to H^{2}_{\eet}(X, \so^{\times})) \to 0 \\ 
&0 \to H^i_{\eet}(X, \so^{\times})/{\rm Im}\, {\partial_i} 
\to H^i_{v}(X, \so^{\times})  \xrightarrow{{\rm HTlog}_i}  {\rm Ker}( {\partial_{i+1}} 
:\Omega^i(X)(-i) \to H^{i+1}_{\eet}(X, \so^{\times})) \to 0  \text{ for all } d \ge i \ge 2.  
\end{align*}

\end{enumerate}  
\end{proposition}

Here the maps ${\partial_{i+1}} 
: \Omega^i(X)(-i) \to H^{i+1}_{\eet}(X, \so^{\times})$ (for $i\ge1$) are the maps given by the composition 
\begin{equation}
\label{eau}
{\partial_{i+1}} 
:\quad \Omega^i(X)(-i) \xleftarrow[\sim]{\rm HT} H^0(X, \R^i\nu_*\so) \xrightarrow[\sim]{\rm exp} H^0(X, \R^i\nu_*\so^{\times}) \lomapr{d_{i+1}}  H^{i+1}_{\eet}(X, \so^{\times}),
\end{equation}
where $d_{i+1}$ is the only differential on the $E_{i+1}$-page of the Leray spectral sequence (see the proof):
\[ E_2^{i,j}= H^i_{\eet}(X, \R^j\nu_{*} \so^{\times}) \Rightarrow H^{i+j}_{v}(X, \so^{\times}) \]

\begin{proof}
We analyze the terms of the above spectral sequence. 

(i) {\em The $E_2$-page.} 
For all $i,j\ge1$, we have isomorphisms
\[ H^i_{\eet}(X, \R^j\nu_{*} \so^{\times}) \xrightarrow{\sim} H^i_{\eet}(X, \Omega^j_X(-j)) \]
and the term on the right is zero since $X$ is Stein. So, on the page $E_2$, the only non-zero terms will be in the row $j=0$ and column $i=0$ (for $0\le j\le d$) and we have: 
\[ E_2^{i,0} = H^i_{\eet}(X, \so^{\times}) \text{ and } E_2^{0,j} =H^0_{\eet}(X,\R^j{\nu_*}\so^{\times})\text{ for all } i\ge 0,d \ge  j \ge 1. \]
There is  only one non-zero differential  $d_2: H^0_{\eet}(X,\R^1{\nu_*}\so^{\times}) \to H^2_{\eet}(X, \so^{\times})$.  

(ii) {\em The $E_3$-page.} The terms $E_3^{i,j}$ are equal to $E_2^{i,j}$ except for $(i,j) \in \{ (2,0), (0,1) \}$ where they are
\[ E_3^{2,0} \simeq H^2_{\eet}(X, \so^{\times})/{\rm Im}\, d_2 \text{ and } E_3^{0,1} \simeq  {\rm Ker}(d_2). \]
There is  only one non-zero differential  $d_3:H^0_{\eet}(X,\R^2{\nu_*}\so^{\times}) \to H^3_{\eet}(X, \so^{\times})$.  

(iii) {\em The $E_{\infty}$-page.}  Iterating the above computation, we get that the spectral sequence degenerates at $d+2$ and we have:
\begin{align*}
&E^{0,0}_{\infty}= H^0_{\eet}(X, \so^{\times}), \quad E^{1,0}_{\infty}= H^1_{\eet}(X, \so^{\times}) \\
&E^{i,0}_{\infty}\simeq H^i_{\eet}(X, \so^{\times})/ {\rm Im}\,d_i,\text{ for }  d+1 \ge i \ge 2, \\
&E^{i,0}_{\infty}\simeq H^i_{\eet}(X, \so^{\times}), \text{ for }  i \ge d+2, \\
&E^{0,j}_{\infty}\simeq {\rm Ker}(d_{j+1}: H^0_{\eet}(X,\R^j{\nu_*}\so^{\times}) \to H^{j+1}_{\eet}(X, \so^{\times})), \text{ for }  d \ge j \ge 1,
\end{align*}
as wanted.
\end{proof} 

\begin{corollary}\label{sally1}
Let $X$ be a smooth Stein space of dimension $d$ over $C$. For $d\ge i \ge 2$, there is an exact sequence: 
\begin{equation}\label{kwak1}
 0 \to {\rm Coker}({\rm HTlog}_{i-1}) \to H^i_{\eet}(X, \so^{\times}) \lomapr{\nu^*_i}  {\rm Ker}({\rm HTlog}_{i}) \to 0, 
 \end{equation}
where $\nu^*_i$ is  the pullback  map $H^i_{\eet}(X, \so^{\times})\to H^i_{v}(X, \so^{\times})$.
\end{corollary}

\begin{proof}
Let $d\ge i \ge 2$.  
The second exact sequence from Proposition~\ref{ses-leray} shows that $\nu^*_i$ factorises through $H^i_{\eet}(X, \so^{\times})/{\rm Im}\,d^{\times}_{i}$ and that its image is equal to ${\rm Ker}({\rm HTlog}_{i})$. Hence we get the surjectivity on the right in \eqref{kwak1}. 

Let us now compute the kernel of $\nu^*_i$. By  Proposition~\ref{ses-leray}, it is equal to the image of the map 
$${\partial_i} 
: \Omega^{i-1}(X)(-i+1) \to H^i_{\eet}(X, \so^{\times}). $$ 
Using the second exact  sequences from Proposition~\ref{ses-leray}  but in degree $i-1$,  we obtain that the kernel of ${\partial_i}$ 
is equal to the image of ${\rm HTlog}_{i-1}$. This gives an isomorphism: 
$${\partial_i} 
:  {\rm Coker}({\rm HTlog}_{i-1})  \xrightarrow{\sim}  {\rm Im}({\partial_i}
), $$
as wanted. 
\end{proof}

\section{Comparison of two boundary maps}
\label{compatibility} 

To compute the image of the Hodge-Tate logarithm, we will relate  it to the kernel of  the boundary morphism appearing in the fundamental diagram from \cite[Th.\,1.8]{CDN3}, \cite[Th.\,6.14]{CN5}. In order to check the compatibility between the morphism ${\rm HTlog}$ and the map from \cite{CN5}, we use the alternative definition of the latter given by Bosco in \cite{Bos21}. We start this section by recalling briefly the construction of the two maps.

\subsection{Hodge-Tate map revisited} 

We express here the Hodge-Tate map as a Poincar\'e Lemma  projection map. 
\subsubsection{Poincar\'e Lemma} 
We first state the Poincar\'e Lemma for the de Rham period sheaves $\mathbb{B}^+_{\dr}$ and $\mathbb{B}_{\dr}$ from \cite{Sch13a}. We start by recalling the definitions of the various period sheaves and some of their properties. We work here on the pro-\'etale site of a locally Noetherian adic space $X$ over ${\rm Spa}(\Q_p,\Z_p)$. 
 
The Fontaine period sheaf ${\bb A}_{\rm inf}$ is defined as the sheaf $W(\so_{\proeet}^{\flat,+})$. It comes with a morphism $\theta: {\bb A}_{\rm inf} \to \so^+$. We write ${\bb B}_{\rm inf}:={\bb A}_{\rm inf}[\tfrac{1}{p}]$. The morphism $\theta$ extends to $\theta : {\bb B}_{\rm inf} \to \so$.  
The de Rham period sheaf
 \[ {\bb B}_{\rm dR}^+:= \lim_n {\bb B}_{\rm inf}/({\rm Ker}(\theta))^n \]
admits a filtration ${\rm Fil}^{i}{\bb B}_{\rm dR}^+:= {\rm Ker}(\theta)^i{\bb B}_{\rm dR}^+$. Let $t$ be a generator of ${\rm Fil}^1{\bb B}_{\rm dR}^+$. We set ${\bb B}_{\rm dR}:={\bb B}_{\rm dR}^+[t^{-1}]$ and equip it with the induced filtration. The morphism $\theta$ induces an isomorphism ${\bb B}_{\rm dR}^+/t \xrightarrow{\sim} \so$.   

For $0<u \le v$, we define ${\bb A}^{[u,v]}$ as the $p$-adic completion of the sheaf ${\bb A}_{\rm inf}[\tfrac{p}{[\alpha]},\tfrac{[\beta]}{p}]$ for elements $\alpha$ and $\beta$ in $\so_{C^{\flat}}$ such that $v(\alpha) =\tfrac{1}{v}$ and $v(\beta) =\tfrac{1}{u}$ and 
\[ {\bb B}:= \lim_{0<u \le v} {\bb A}^{[u,v]}. \] 
We have the relative fundamental exact sequence of $p$-adic Hodge theory:
\begin{align}
\label{sesB}
&0 \to \Q_p \to \mathbb{B}[\tfrac{1}{t}]^{\varphi=1} \to  \mathbb{B}_{\rm dR}/\mathbb{B}_{\rm dR}^+ \to 0 \\
&0 \to \Q_p(r) \to \mathbb{B}^{\varphi=p^r} \to  \mathbb{B}_{\rm dR}^+/t^r\mathbb{B}_{\rm dR}^+  \to 0, \notag
\end{align}
for all $r \ge 1$. We note that the map $\mathbb{B}^{\varphi=p} \to  \mathbb{B}_{\rm dR}^+/t\mathbb{B}_{\rm dR}^+ \xrightarrow[\sim]{\theta} \so$ can be identified, via the identification of $\mathbb{B}^{\varphi=p} $ with the $C$-points of the universal cover of the multiplicative $p$-divisible group,  with the composition (see~\cite[Prop.\,2.20, Rem.\,2.21, Ex.\,2.22]{LB18}): 
\begin{equation}
\label{theta-log}
\lim_{x\mapsto x^{p}} (1+ {\frak m} \so)=U^{\flat} \xrightarrow{x \mapsto x^{\sharp}} U \xrightarrow{\log} \so, 
\end{equation} 
where the first map is the sharp map given by the projection on the first factor. 

Similarly, for a smooth adic space $X$ over $K$, we define $\so\mathbb{B}_{\rm inf}:= \tilde{\mu}^*\so_{\eet} \otimes_{W(k)} \mathbb{B}_{\rm inf}$, where  $\tilde{\mu}$ is the canonical projection $\tilde{\mu}: X_{\proeet} \to X_{\eet}$. 
We still have a map $\theta: \so\mathbb{B}_{\rm inf} \to \so_{\proeet}$. Then we set: 
\[ \so \mathbb{B}_{\dr}^+:= \lim_n \so{\mathbb B}_{\rm inf}/{\rm Ker}(\theta)^n, \quad   {\rm Fil}^r\so\mathbb{B}_{\dr}^+:= {\rm Ker}(\theta)^r\so\mathbb{B}_{\dr}^+. \]
Finally, for a generator $t$ of  ${\rm Fil}^1\mathbb{B}_{\dr}^+$, we take  $\so\mathbb{B}_{\dr}^+[t^{-1}]$ and  equip it with the filtration induced from $\so\mathbb{B}_{\dr}^+$. 
Let $\so\mathbb{B}_{\dr}$ be the completion of $\so\mathbb{B}_{\dr}^+[t^{-1}]$ with respect to this filtration. 
\footnote{Note that we are not using here the original definition of $\so\mathbb{B}_{\dr}$ (which was not completed with respect to the filtration), but the one due to \cite[Def.\,2.2.10, Rem.\,2.2.11]{DLLZ22}.} 

\begin{theorem}[Poincar\'e Lemma] \cite[Cor.\,6.13]{Sch13a}, \cite[Cor.\,2.4.2]{DLLZ22}
\label{PL-sch}
Let $X$ be a smooth rigid space of dimension $d$ over $K$.  Then:
\begin{enumerate}
\item  There are exact sequences of pro-\'etale sheaves on $X$: 
\begin{align}\label{ciemno1}
&0 \to \mathbb{B}^+_{\dr} \to \so\mathbb{B}^+_{\dr} \stackrel{\nabla}{\to} \so\mathbb{B}^+_{\dr} \otimes_{\tilde{\mu}^*\so} \Omega^1 \stackrel{\nabla}{\to} \cdots \stackrel{\nabla}{\to} \so\mathbb{B}^+_{\dr} \otimes_{\tilde{\mu}^*\so} \Omega^d \to 0, \\
&0 \to {\rm Fil}^r\mathbb{B}^+_{\dr} \to {\rm Fil}^r\so\mathbb{B}^+_{\dr} \stackrel{\nabla}{\to}  {\rm Fil}^{r-1}\so\mathbb{B}^+_{\dr} \otimes_{\tilde{\mu}^*\so} \Omega^1 \stackrel{\nabla}{\to}  \cdots \stackrel{\nabla}{\to} {\rm Fil}^{r-d}\so\mathbb{B}^+_{\dr} \otimes_{\tilde{\mu}^*\so} \Omega^d \to 0\notag
\end{align}
for all $r\in \Z$, where $\Omega^i:=\tilde{\mu}^*\Omega^i_{\eet}$ for $i\ge1$.  (We note that the notation here is slightly different from the one of \cite{Sch13a} since the pro-\'etale sheaf that we denote by $\so_{\proeet}$ is the completion of $\tilde{\mu}^*\so_{\eet}$.)  We have analogues of the exact sequences in \eqref{ciemno1}    for $\mathbb{B}_{\dr}$ and $\so\mathbb{B}_{\dr} $. 
\item For $r\in\Z$, the quotient complex
$$
0 \to \gr^r_{F}\mathbb{B}_{\dr} \to \gr^r_{F}\so\mathbb{B}_{\dr} \stackrel{\nabla}{\to} \gr^{r-1}_{F}\so\mathbb{B}_{\dr} \otimes_{\tilde{\mu}^*\so} \Omega^1 \stackrel{\nabla}{\to} \cdots \stackrel{\nabla}{\to} \gr^{r-d}_{F}\so\mathbb{B}_{\dr} \otimes_{\tilde{\mu}^*\so} \Omega^d \to 0
$$
is exact and can be identified with the complex 
$$
0 \to \so(r) \to \ \so\mathbb{C}(r) \stackrel{\nabla}{\to}  \so\mathbb{C}(r)\otimes_{\tilde{\mu}^*\so} \Omega^1(-1) \stackrel{\nabla}{\to} \cdots \stackrel{\nabla}{\to} \so\mathbb{C}(r) \otimes_{\tilde{\mu}^*\so} \Omega^d(-d) \to 0,
$$
where we set $ \so\mathbb{C}:=\gr^0_F\so{\mathbb B}_{\dr}$. We note that, by \cite[Ch.\,6]{Sch13a}, $\gr^i_F\so{\mathbb B}_{\dr}\simeq \so\mathbb{C}(i)$.
\end{enumerate}
\end{theorem}

  We denote by $\epsilon$ the map $\gr^0_{F}\mathbb{B}_{\dr} \to \gr^0_{F}\so\mathbb{B}_{\dr}$.  Thus we have the exact sequence
$$
0\to \so\lomapr{\epsilon}  \gr^0_{F}\so\mathbb{B}_{\dr}\to \coker \varepsilon\to 0.
$$
This sequence projects to the exact sequence on $X_{C,\eet}$ (here $\mu: X_{C,\proeet}\to X_{C,\eet}$ is the canonical projection)
\begin{equation}\label{tea12}
0\to \mu_*(\so)\lomapr{\epsilon}  \mu_*(\gr^0_{F}\so\mathbb{B}_{\dr})\to \mu_*(\coker \varepsilon)\to \R^1\mu_*(\so).
\end{equation}
 Since the maps
$$
\mu_*(\coker \varepsilon)\lomapr{\nabla} \mu_*( \so\mathbb{C}\otimes_{\tilde{\mu}^*\so} \Omega^1(-1))\stackrel{}{\leftarrow} \Omega^1_{X_C}(-1)
$$
are isomorphisms (see the proof of Lemma \ref{sniad3}(\ref{sniad3-1}) for details),  they yield 
an isomorphism $\beta: {\mu}_*(\coker \varepsilon) \stackrel{\sim}{\to} \Omega_{X_C}^1(-1).$ 
We denote by ${\rm PL}^{-1}$ the composition of $\beta^{-1}$ and the connecting morphism in \eqref{tea12}: 
\[ {\rm PL}^{-1}: \Omega_{X_C}^1(-1) \to \R^1 {\mu}_* (\mathbb{B}^+_{\dr}/t). \]
It will follow from  Lemma \ref{sniad3}(\ref{sniad3-1}) below that ${\rm PL}^{-1}$ is an isomorphism and we write ${\rm PL}$ for  its inverse.  

The above construction can be made more explicit. Theorem~\ref{PL-sch} has the following useful corollary (see \cite[Rem.\,3.18]{LB18} or \cite[(6.4), (6.5)]{Bos21} for a more general statement): 

\begin{proposition}\label{revis1}
Let $X$ be a smooth rigid space of dimension $d$ over $K$ and let $r\in\Z$. Then we have the following quasi-isomorphisms on $X_{C, \eet}$ (the differentials are $\B_{\dr}$-linear):\begin{align*}
&\R{\mu}_*(\mathbb{B}_{\dr}) \stackrel{\sim}{\leftarrow} \big(\so {\otimes}^{\Box}_K \bdr \stackrel{d}{\to} \Omega^1 {\otimes}^{\Box}_K \bdr \stackrel{d}{\to} \cdots \stackrel{d}{\to} \Omega^d {\otimes}^{\Box}_K \bdr  \big) \notag \\
&\R{\mu}_*({\rm Fil}^r\mathbb{B}_{\dr})\stackrel{\sim}{\leftarrow}\big(\so {\otimes}^{\Box}_K {\rm Fil}^r\bdr \stackrel{d}{\to }\Omega^1 {\otimes}^{\Box}_K {\rm Fil}^{r-1}\bdr  \stackrel{d}{\to} \cdots \stackrel{d}{\to} \Omega^d {\otimes}^{\Box}_K {\rm Fil}^{r-d}\bdr  \big).
\end{align*}
\end{proposition}
 Here the ${\Box}$ refers to the solid tensor product defined by Clausen-Scholze, see~\cite{Sch19} and we regard $\Omega^i\otimes^{\Box}_K {\rm Fil}^j\bdr$ as sheaves on $X_{C,\eet}$ via the equivalence $X_{C,\eet}\simeq \lim_{[K^{\prime}:K]<\infty}X_{K^{\prime},\eet}$ (see \cite[Sec.\,3.2]{LB18} for more details).
 The  quasi-isomorphisms in Proposition \ref{revis1} are topological, i.e., more specifically, they are quasi-isomorphisms in the $\infty$-derived category of sheaves with values in solid $K$-modules. We will denote this category by $\sd(X_{\eet},K_{\Box})$.
In particular, it follows from this proposition, that the pro-\'etale cohomologies of $\mathbb{B}_{\dr}/\mathbb{B}_{\dr}^+$ and $\mathbb{B}^+_{\dr}/t$ are computed by the following complexes on $X_{C,\eet}$: 
\begin{align}
\label{poincareBdr}
&\R{\mu}_*(\mathbb{B}_{\dr}/\mathbb{B}_{\dr}^+) \stackrel{\sim}{\leftarrow} \big(\so {\otimes}^{\Box}_K (\bdr/\bdr^+)\stackrel{d}{\to} \Omega^1 {\otimes}^{\Box}_K( \bdr/t^{-1}\bdr^+) \stackrel{d}{\to} \cdots \stackrel{d}{\to} \Omega^d {\otimes}^{\Box}_K (\bdr/t^{-d}\bdr^+)  \big), \\
&\R{\mu}_*(\mathbb{B}^+_{\dr}/t) \stackrel{\sim}{\leftarrow} \big(\so_{X_C} \xrightarrow{0} \Omega_{X_C}^1(-1) \xrightarrow{0} \cdots \xrightarrow{0} \Omega_{X_C}^d(-d)  \big). \notag
\end{align}
The maps
\begin{equation}\label{sniad33}
{\rm PL}: \quad \R{\mu}_*(\mathbb{B}^+_{\dr}/t)\to  \Omega^i_{X_C}(-i)[-i],\quad {\rm PL}: \quad H^i_{\proeet}(X_C,\mathbb{B}^+_{\dr}/t)\to  \Omega_{X_C}^i(X_C)(-i)
\end{equation}
are given by  the canonical projections.

\subsubsection{Stein spaces}\label{bStein}

We now apply the above computations to the case when $X$ is a Stein space defined over $K$.  
The quasi-isomorphism  \eqref{poincareBdr} yields the quasi-isomorphism 
\begin{align}\label{sniad2}
\R\Gamma_{\proeet}(X_C,\mathbb{B}^+_{\dr}/t) \simeq \big(\so_{X_C}(X_C) \xrightarrow{0} \Omega_{X_C}^1(X_C) (-1) \xrightarrow{0} \cdots \xrightarrow{0} \Omega_{X_C}^d(X_C) (-d)  \big). 
\end{align}

The case of $\mathbb{B}_{\dr}/\mathbb{B}_{\dr}^+$ is a bit subtler because of the usual problems with topological tensor products and limits, respectively colimits. Let $\{X_n\},\,n\in\N$, be a strictly increasing dagger affinoid covering of $X$ (i.e., we have $X_n\Subset X_{n+1}$: the adic closure of $X_n$ is contained in $X_{n+1}$). Then we have the following lemma: 

\begin{lemma}
With the above notations, there is an exact sequence
\begin{equation}\label{sleepy3}
0\to \lim_n(H^i_{\dr}(X_n)\otimes^{\Box}_K(\bdr/t^{-i}\bdr^+))\to H^i_{\proeet}(X_{C},\mathbb{B}_{\dr}/\mathbb{B}_{\dr}^+)\xrightarrow{\pi}  (\Omega_{X_C}^i(X_C)/{\rm Ker}\, d)(-i-1)\to 0. 
\end{equation}
\end{lemma}

\begin{proof}
The dagger analogue of  \eqref{poincareBdr} yields a quasi-isomorphism
\begin{align}\label{small1}
\R\Gamma_{\proeet}(X_{n,C},\mathbb{B}_{\dr}/\mathbb{B}_{\dr}^+) \simeq& \big(\so(X_n) {\otimes}^{\Box}_K (\bdr/\bdr^+) \to \Omega^1(X_n) {\otimes}^{\Box}_K (\bdr/t^{-1}\bdr^+) \to \cdots\\\nonumber
&\qquad\quad\cdots \to \Omega^d(X_n) {\otimes}^{\Box}_K (\bdr/t^{-d}\bdr^+)  \big).
\end{align}
Note that the left hand side in~\eqref{small1} is computed via the formula from \cite[Lem.\,3.23]{Bos23}.  
Passing to the limit over $n$ we obtain a quasi-ismorphism
\begin{align*}
\R\Gamma_{\proeet}(X_{C},\mathbb{B}_{\dr}/\mathbb{B}_{\dr}^+) \simeq& \R\lim_n\big(\so(X_n) {\otimes}^{\Box}_K (\bdr/\bdr^+) \to \Omega^1(X_n) {\otimes}^{\Box}_K (\bdr/t^{-1}\bdr^+) \to \cdots\\
&\qquad\quad\cdots \to \Omega^d(X_n) {\otimes}^{\Box}_K (\bdr/t^{-d}\bdr^+)  \big).
\end{align*}
We used here the fact that $\R\Gamma_{\proeet}(X_{C},\mathbb{B}_{\dr}/\mathbb{B}_{\dr}^+)\stackrel{\sim}{\to}\R\lim_n\R\Gamma_{\proeet}(X_{n,C},\mathbb{B}_{\dr}/\mathbb{B}_{\dr}^+)$. 

For $i,n \geq 0$, we have
\begin{align*}
& H^i_{\proeet}(X_{n,C},\mathbb{B}_{\dr}/\mathbb{B}_{\dr}^+) ={\rm Ker }\,d_i/{\rm Im}\, d_{i-1},\\
&  \Omega^{i-1}(X_n) {\otimes}^{\Box}_K (\bdr/t^{-i+1}\bdr^+) \lomapr{d_{i-1}} \Omega^i(X_n) {\otimes}^{\Box}_K (\bdr/t^{-i}\bdr^+) \lomapr{d_i} \Omega^{i+1}(X_n) {\otimes}^{\Box}_K (\bdr/t^{-i-1}\bdr^+). 
\end{align*}
This yields an  exact sequence and an isomorphism
\begin{align*}
& 0\to \Omega^i(X_n)^{d=0}\otimes^{\Box}_K(\bdr/t^{-i}\bdr^+)\to {\rm Ker}\, d_i\to (\Omega_{X_{n,C}}^i(X_{n,C})/{\rm Ker}\,d)(-i-1)\to 0,\\
& {\rm Im }\,d_{i-1}\simeq {\rm Im}\,d\otimes^{\Box}_K(\bdr/t^{-i}\bdr^+).
\end{align*}
Putting them together we obtain  the exact sequence
\begin{equation}\label{sleepy1}
0\to H^i_{\dr}(X_n)\otimes^{\Box}_K(\bdr/t^{-i}\bdr^+)\to H^i_{\proeet}(X_{n,C},\mathbb{B}_{\dr}/\mathbb{B}_{\dr}^+)\to  (\Omega_{X_{n,C}}^i(X_{n,C})/{\rm Ker}\, d)(-i-1)\to 0
\end{equation}
We note here that $H^i_{\dr}(X_n)$ is of finite rank over $K$ (see~\cite{GK02}). 
Passing to the limit over $n$ we get the exact sequence
$$
0\to \lim_n(H^i_{\dr}(X_n)\otimes^{\Box}_K(\bdr/t^{-i}\bdr^+))\to\lim_n H^i_{\proeet}(X_{n,C},\mathbb{B}_{\dr}/\mathbb{B}_{\dr}^+)\to  (\Omega_{X_C}^i(X_C)/{\rm Ker}\, d)(-i-1)\to 0.
$$
The exactness on the right follows from the fact that $\R^1 \lim_n(H^i_{\dr}(X_n)\otimes^{\Box}_K(\bdr/t^{-i}\bdr^+))=0$ because the pro-system $\{ H^i_{\dr}(X_n)\otimes^{\Box}_K(\bdr/t^{-i}\bdr^+)\}_{n\in\N}$ is topologically Mittag-Leffler. We note that, 
since $\R^1\lim_n (\Omega_{X_{n,C}}^i(X_{n,C})/{\rm Ker}\, d)=0$, we also have 
$$
H^i_{\proeet}(X_{C},\mathbb{B}_{\dr}/\mathbb{B}_{\dr}^+)\stackrel{\sim}{\to}H^i(\R\lim_n \R\Gamma_{\proeet}(X_{n,C},\mathbb{B}_{\dr}/\mathbb{B}_{\dr}^+))\simeq \lim_n H^i_{\proeet}(X_{n,C},\mathbb{B}_{\dr}/\mathbb{B}_{\dr}^+).
$$
This concludes the proof of the lemma. 
\end{proof}

\begin{remark}\label{overC1}
If $X$ is a variety defined over $C$ that does not necessarily come from the base change of a variety defined over $K$, we still get the maps ${\rm PL} $ and  $\pi$ from~\eqref{sniad33} and~\eqref{sleepy3}: 
\begin{align*}
{\rm PL}: & \quad H^i_{\proeet}(X,\mathbb{B}^+_{\dr}/t)\to  \Omega^i(X)(-i),\\
\pi:  &\quad H^i_{\proeet}(X,\mathbb{B}_{\dr}/\mathbb{B}_{\dr}^+)  \to (\Omega^i(X)/{\rm Ker}\, d)(-i-1).
\end{align*}
 Indeed, for a covering $\{X_n\}_n$ as above, since the $X_n$'s are dagger affinoids, they are defined over finite extensions $K_n$ of $K$ and we still have the exact sequence~\eqref{sleepy1} (replacing $K$ by the $K_n$ for each $n\in \N$). By taking  limits over $n$, we obtain the map $\pi$:
\begin{align*}
\pi:  H^i_{\proeet}(X,\mathbb{B}_{\dr}/\mathbb{B}_{\dr}^+) & \to \lim_nH^i_{\proeet}(X_{n},\mathbb{B}_{\dr}/\mathbb{B}_{\dr}^+)\to \lim_n (\Omega^i(X_{n})/{\rm Ker}\, d)(-i-1)\\
 & \stackrel{\sim}{\leftarrow} (\Omega^i(X)/{\rm Ker}\, d)(-i-1).
 \end{align*}
 
 Similarly for the map ${\rm PL}$. We have a dagger analogue of \eqref{sniad2}: 
 \begin{align*}
\R\Gamma_{\proeet}(X_{n},\mathbb{B}^+_{\dr}/t) \simeq \big(\so(X_{n,K_n}) {\otimes}^{\Box}_{K_n} C \xrightarrow{0} \Omega^1(X_{n,K_n}) {\otimes}^{\Box}_{K_n} C(-1) \xrightarrow{0} \cdots \xrightarrow{0} \Omega^d(X_{n,K_n}) {\otimes}^{\Box}_{K_n} C(-d)  \big). 
\end{align*}
This yields the maps ${\rm PL}_n:  H^i_{\proeet}(X_n,\mathbb{B}^+_{\dr}/t)\to  \Omega^i(X_n)(-i)$. Passing to the limit over $n$, we get the map
$$
{\rm PL}:  \quad H^i_{\proeet}(X,\mathbb{B}^+_{\dr}/t)\to \lim_nH^i_{\proeet}(X_n,\mathbb{B}^+_{\dr}/t)\lomapr{{\rm PL}_n}  \lim_n\Omega^i(X_n)(-i)\stackrel{\sim}{\leftarrow} \Omega^i(X)(-i).
$$
\end{remark}  

\subsubsection{Hodge-Tate morphism revisited}
\label{faltings-ext}

Let $X$ be a smooth rigid analytic space over $K$. 
A consequence of Theorem~\ref{PL-sch}  is the following short exact sequence of pro-\'etale sheaves (called Faltings extension) on $X_{\proeet}$: 
\[ 0 \to \so(1) \to {\rm gr}^1_F\so\mathbb{B}_{\dr}^+ \to \so \otimes_{\tilde{\mu}^*\so} \Omega^1 \to 0, \]
which yields the boundary map 
\begin{equation}\label{break1}
 \so \otimes_{\tilde{\mu}^*\so} \Omega^1\to  \so(1)[1].
\end{equation}
Then the Hodge-Tate morphism ${\rm HT}$ from Proposition~\ref{omega}  in degree $1$ is given  
as the twist by $(-1)$ 
of the inverse of the projection of the map \eqref{break1} from the pro-\'etale site of $X_C$ to the \'etale site of $X_C$ 
$$
\partial_{\B^+}:\quad
\Omega_{X_C}^1\stackrel{\sim}{\to}  \R^1{{\mu}_*}\so(1).
$$
Via  wedge and cup products the map $\partial_{\B^+}$ induces a quasi-isomorphism
$$
\partial^i_{\B^+}:\quad
\Omega_{X_C}^i\stackrel{\sim}{\to}  \R^i{{\mu}_*}\so(i).
$$
Here we use the fact that $\wedge^i \R^1{{\mu}_*}\so(1) \stackrel{\sim}{\to}\R^i{{\mu}_*}\so(i)$  (see the proof of \cite[Lem.\,3.24]{Sch13b}).  The Hodge-Tate morphism ${\rm HT}$ in degree $i$  is given by the $(-i)$-twist of the inverse of $\partial^i_{\B^+}$.

We have the following result: 

\begin{lemma}\label{sniad3}
\begin{enumerate}
\item\label{sniad3-1} Let $X$ be a smooth rigid analytic space over $K$.  Let $i\geq 1$. 
There is a commutative diagram of sheaves on $X_{C,\eet}$
\[ \xymatrix{ \R^i\mu_* \so \ar[r]^{\rm HT} & \Omega^i_{X_C}(-i)   \ar[dl]^{{\rm PL}^{-1}}\\
\R^i\mu_*( \mathbb{B}_{\dr}^+/t) \ar[u]^{\theta}_{\wr}}. \] 
In particular, ${\rm PL}^{-1}$ is an isomorphism. 
\item\label{sniad3-2} Let $X$ be a smooth Stein  space over $C$.  Let $i\geq 1$. 
There is a commutative diagram 
\[ \xymatrix{ H^i_{\proeet}(X, \so) \ar[r]^{\rm HT}_{\sim} & \Omega_X^i(X)(-i)   \\
H^i_{\proeet}(X, \mathbb{B}_{\dr}^+/t) \ar[u]^{\theta}_{\wr}. \ar[ur]_{{\rm PL}} 
}\] 
In particular, ${\rm PL}$ is an isomorphism. 
\end{enumerate}
\end{lemma}

\begin{proof}We start with the first claim.  
Identifying ${\mathbb B}^+_{\dr}/t$ and $\so$ via $\theta$, we need to show that ${\rm PL}^{-1}{\rm HT}={\rm Id}$ or that ${\rm PL}^{-1}={\rm HT}^{-1}$. This is just a question of unwinding some definitions. 

 Let $i=1$.   
Consider the canonical  map of exact sequences on $X_{\proeet}$: 
$$
\xymatrix{
 0 \ar[r] &  \so(1) \ar[r] \ar@{=}[d]&  {\rm gr}^1_F\so\mathbb{B}_{\dr}^+ \ar[r]^-{\nabla} \ar[d] &  \so \otimes_{\tilde{\mu}^*\so} \Omega^1 \ar[r] \ar[d] &  0\\
 0 \ar[r] &  \so(1) \ar[r]^-{\epsilon_1} &  {\rm gr}^1_F\so\mathbb{B}_{\dr} \ar[r]^-{\nabla} &  \so{\mathbb C} \otimes_{\tilde{\mu}^*\so} \Omega^1,
}
$$
where the first sequence is obtained from the ${\mathbb B}^+_{\dr}$-Poincar\'e Lemma and the second one from its $\mathbb{B}_{\dr}$-version. 
By projecting the above diagram to $X_{C, \eet}$ we obtain a map of exact sequences
$$
\xymatrix{
 0 \ar[r] & {\mu}_*\so(1) \ar[r] \ar@{=}[d]&  {\mu}_*{\rm gr}^1_F\so\mathbb{B}_{\dr}^+ \ar[r]^{\nabla} \ar[d] &  \Omega_{X_C}^1 \ar[r]^-{\partial_{\B^+}={\rm HT}^{-1}(1)} \ar[d]^{f_1}_{\wr}&  \R^1{\mu}_*\so(1)\ar@{=}[d]  \ar[r] & 0\\
 0 \ar[r] & {\mu}_*\so(1) \ar[r]^-{\epsilon_1} & {\mu}_* {\rm gr}^1_F\so\mathbb{B}_{\dr} \ar[r] \ar[rd]^{\nabla} & \mu_*\coker(\epsilon_1)\ar[d]^{f_2}_{\wr}\ar[r]^-{\partial_{\B}} &  \R^1{\mu}_*\so(1)\ar[r] & 0\\
  &  & &   \mu_*(\so{\mathbb C} \otimes_{\tilde{\mu}^*\so} \Omega^1) 
}
$$
Since, by \cite[Prop.\,6.16]{Sch13a},  $\R\mu_*{\rm gr}^i_F\so\mathbb{B}_{\dr}\simeq \so(i)$,  the second row is exact on the right.  The map $f_2$, a priori just injective, is an isomorphism because  the composition $f_2f_1$ is an isomorphism; it follows that the map $f_1$ is also an isomorphism. The diagram shows that $\partial_{\B}={\rm HT}^{-1}(1)f_1^{-1}$. 
   
   Consider now the following map of exact sequences defined on $X_{C,\proeet}$ (obtained from the $\mathbb{B}_{\dr}$-Poincar\'e Lemma)
   $$
\xymatrix{
 0 \ar[r] &  \so(1) \ar[r]^-{\epsilon_1} &  {\rm gr}^1_F\so\mathbb{B}_{\dr} \ar[r]^-{\nabla} &  \so{\mathbb C} \otimes_{\tilde{\mu}^*\so} \Omega^1 \\
  0 \ar[r] &  \so \ar[r]^-{\epsilon} \ar[u]^{t}&  {\rm gr}^0_F\so\mathbb{B}_{\dr} \ar[r]^-{\nabla} \ar[u]^{t}&  \so{\mathbb C} \otimes_{\tilde{\mu}^*\so} \Omega^1(-1).  \ar[u]^{t}
}
$$
 By projecting it to  $X_{C,\eet}$ we obtain a map of exact sequences
$$
\xymatrix{
 0 \ar[r] &  {\mu}_*\so(1) \ar[r]^-{\epsilon_1} & {\mu}_* {\rm gr}^1_F\so\mathbb{B}_{\dr} \ar[r] &\mu_*\coker(\epsilon_1)\ar[r]^-{\partial_{\B}} &  \R^1 {\mu}_*\so(1) \ar[r] & 0\\
   0 \ar[r] &    {\mu}_*\so \ar[r]^-{\epsilon} \ar[u]^{t}_{\wr}&   {\mu}_*{\rm gr}^0_F\so\mathbb{B}_{\dr} \ar[r] \ar[u]^{t}_{\wr}&  \mu_*\coker(\epsilon)\ar[u]^{t}_{\wr}\ar[r]^{{\rm PL}^{-1}\beta}\ar[d]^{\beta}  &  \R^1 {\mu}_*\so. \ar[u]^{t}_{\wr}\\
& & & \Omega^1_{X_C}(-1)\ar[ru]_{{\rm PL}^{-1}}
}
$$
We note that the map $\beta$ is an isomorphism because so is the map $f_1$ (and $\beta=f^{-1}_1(-1)$). 

 In total, the  above diagrams show that
$$
{\rm PL}^{-1}\beta(1)=\partial_{\B}={\rm HT}^{-1}(1)f_1^{-1},
$$
hence ${\rm PL}^{-1}={\rm HT}^{-1}$,
as wanted. 
This proves the first claim of our lemma for $i=1$. The case of $i\ge 1$ is obtained by taking wedge products. 

 For the second claim of the lemma, choose a Stein covering of $X$ by Stein spaces $\{X_n\}$ such that each $X_n$ is defined over a finite extension $K_n$ of $K$  (to do that you may start with a Stein affinoid covering and then take the naive interiors of these affinoids containing the previous affinoids). The wanted diagram is obtained by taking the limit over $n$ of the diagram in claim (1) (note that  $\R^1\lim_n$ is trivial for all the terms of the diagram). 
\end{proof}

\subsection{The Bloch-Kato exponential}

We restrict our attention now to smooth Stein spaces. 
We will introduce here the Bloch-Kato exponential 
and show how it can be obtained, 
via the filtered ${\mathbb B}_{\dr}$-Poincar\'e Lemma, 
from a boundary map induced by a fundamental exact sequence. 
\subsubsection{The definition of the map ${\rm Exp}$}\label{sleepy2}

We first recall how the geometric  $p$-adic pro-\'etale cohomology of Stein spaces can be computed. In \cite[Th.\,1.8]{CDN3}, \cite[Th.\,6.14]{CN5}, Colmez-Dospinescu-Nizio{\l} proved the following theorem\footnote{Note that in~\cite{CDN3},\cite{CN5}, the set-up is slightly different since they do not work with condensed mathematics. The construction from~\cite{Bos23} however, uses the solid formalism.}: 

\begin{theorem}
Let $X$ be a Stein smooth rigid analytic space over $C$. For  $i \ge 0$, there is a map of exact sequences in $\sd(\Q_{p,\Box})$: 
 \begin{equation}
 \label{pStein}
 \xymatrix@R=5mm@C=6mm{
 0\ar[r] &  \Omega^{i-1}(X)/{\rm Ker}\,d\ar[r]^{{\rm Exp}} \ar@{=}[d] & H_{\proeet}^i(X,\Q_p(i))\ar[r] \ar[d]^{d{\rm Log}} & 
  (H^i_{{\rm HK}}(X){\otimes}^{\Box}_{\breve{F}}\wh{\B}^+_{\st})^{N=0,\phi=p^i}\ar[d]^{\iota_{\rm HK}\otimes\theta} \ar[r] & 0\\
  0\ar[r] &  \Omega^{i-1}(X)/{\rm Ker}\,d\ar[r]^-{d} & \Omega^i(X)^{d=0}\ar[r] & H^i_{\dr}(X)\ar[r] & 0.}
 \end{equation} 
 If $X$ is defined over $K$, this map is Galois equivariant.
\end{theorem}

\begin{remark}
The maps in diagram \eqref{pStein} were constructed using a comparison of $p$-adic pro-\'etale cohomology with syntomic cohomology of Bloch-Kato type. We call  the map ${\rm Exp}$ the ``Bloch-Kato exponential''; this  is supposed to suggest the Bloch-Kato exponential map from \cite{BK}.
\end{remark}
The cohomology $H^r_{\hk}(X)$ appearing on the right of the first exact sequence is the Hyodo-Kato cohomology as defined by Colmez-Nizio{\l} in \cite[Sec.\,4]{CN4}
(it is built from the logarithmic crystalline cohomology $\rg_{\rm cr}(\mathcal{X}_{\so_L,0}/W(k_L)^0)$ for $L/K$ a finite extension with residue field $k_L$, where  $\mathcal{X}_{\so_L}\to {\rm Spf}(\so_L)$ is a semistable formal scheme and $W(k_L)^0$ denotes the formal scheme ${\rm Spf}(W(k_L))$ equipped with the log-structure induced by $\N \to W(k_L), 1 \mapsto 0$).  It is a $(\phi, N, \sg_K)$-module over $\breve{F}$  equipped with a Hyodo-Kato isomorphism
$\iota_{\hk}: H^i_{{\rm HK}}(X)\otimes^{\Box}_{\breve{F}}C\stackrel{\sim}{\to}H^i_{\dr}(X)$.

An alternative construction of diagram \eqref{pStein}  was given by Bosco\footnote{Bosco's construction is given for Stein spaces that are base change to $C$ of varieties defined over $K$ but as we will see here, the result is still valid when it is not the case.} in \cite[Th.\,7.7]{Bos23}. His construction is closely related to the subject of this paper and we will now briefly describe how to get the top row in \eqref{pStein}\footnote{We did not check that the map ${\rm Exp}$ constructed in \cite{Bos23} is the same as the one constructed in \cite{CN5} but we will not need it.}. The starting point is the exact sequence 
 $$
 0\to \Q_p\to {\mathbb B}_e\to {\mathbb B}_{\dr}/{\mathbb B}^+_{\dr}\to 0
 $$
 of pro-\'etale sheaves on $X$, where we set ${\mathbb B}_e:={\mathbb B}[1/t]$. 
 It yields an exact sequence
 $$
 H^{i-1}_{\proeet}(X,{\mathbb B}_e) \lomapr{\alpha_{i-1}}   H^{i-1}_{\proeet}(X,{\mathbb B}_{\dr}/{\mathbb B}^+_{\dr})\to  H^i_{\proeet}(X,\Q_p)\to H^i_{\proeet}(X,{\mathbb B}_e)\lomapr{\alpha_i}  H^i_{\proeet}(X,{\mathbb B}_{\dr}/{\mathbb B}^+_{\dr})
 $$
 Because of limit considerations it is better to work with the dagger analogue of the above exact sequence. Let $\{X_n\},\,n\in\N$, be a strictly increasing dagger affinoid covering of $X$. Each affinoid $X_n$ is the base change to $C$ of an affinoid $X_{n, K_n}$ defined over a finite extension $K_n$ of $K$. 
 
 For $n\in\N$,  we have an exact sequence
$$
 H^{i-1}_{\proeet}(X_{n},{\mathbb B}_e) \lomapr{\alpha_{i-1}}   H^{i-1}_{\proeet}(X_{n},{\mathbb B}_{\dr}/{\mathbb B}^+_{\dr})\to  H^i_{\proeet}(X_{n},\Q_p)\to H^i_{\proeet}(X_{n},{\mathbb B}_e)\lomapr{\alpha_i}  H^i_{\proeet}(X_{n},{\mathbb B}_{\dr}/{\mathbb B}^+_{\dr}).
$$

 Since, by \cite[Th.\,4.1]{Bos23},  \cite[the proof of Th.\,7.7]{Bos23},  and \eqref{sleepy1}, we have an isomorphism and an exact sequence
\begin{align*}
  & H^i_{\proeet}(X_{n},{\mathbb B}_e)\simeq  (H^i_{{\rm HK}}(X_{n}){\otimes}^{\Box}_{\breve{F}}\wh{\B}^+_{\st}[\tfrac{1}{t}])^{N=0,\phi=1},\\
 & 0\to H^i_{\dr}(X_{n,K_n})\otimes^{\Box}_{K_n}(\B_{\dr}/t^{-i}\B^+_{\dr})\to  H^i_{\proeet}(X_{n},{\mathbb B}_{\dr}/{\mathbb B}^+_{\dr})\to (\Omega^i(X_{n})/{\rm Ker}\,d)(-i-1)\to 0,
\end{align*}
it suffices to show that  the map $\alpha_{i-1}$ surjects onto 
$ H^i_{\dr}(X_{n,K_n})\otimes^{\Box}_{K_n}(\B_{\dr}/t^{-i}\B^+_{\dr})$ and 
\begin{align*}
 {\rm Ker}\, \alpha_{i}(i) & \simeq   (H^i_{{\rm HK}}(X_{n}){\otimes}^{\Box}_{\breve{F}}\B_{\log})^{N=0,\phi=p^i},\\
  (H^i_{{\rm HK}}(X_{n}){\otimes}^{\Box}_{\breve{F}}{\B}_{\log})^{N=0,\phi=p^i} & \simeq  (H^i_{{\rm HK}}(X_{n}){\otimes}^{\Box}_{\breve{F}}\wh{\B}^+_{\st})^{N=0,\phi=p^i}.
\end{align*}
 But this follows from the analysis of the slopes of Frobenius on the Hyodo-Kato cohomology. 
 
For all $n\geq 0$, we have constructed compatible  exact sequences
\begin{equation}\label{sniad1}
 \xymatrix@R=5mm@C=6mm{
 0\ar[r] &  \Omega^{i-1}(X_{n})/{\rm Ker}\,d\ar[r]^{{\rm Exp}}  & H_{\proeet}^i(X_{n},\Q_p(i))\ar[r]  & 
  (H^i_{{\rm HK}}(X_{n}){\otimes}^{\Box}_{\breve{F}}\wh{\B}^+_{\st})^{N=0,\phi=p^i}\ar[r] & 0.
  }
\end{equation}
  We obtain the top row in \eqref{pStein} by passing to the limit over $n$ and using the isomorphism 
  $H_{\proeet}^i(X,\Q_p(i))\stackrel{\sim}{\to}\lim_nH_{\proeet}^i(X_{n},\Q_p(i))$.

\subsubsection{Comparison of two boundary maps}

The purpose of this section is to prove the following comparison result: 
\begin{proposition}
\label{ajin}
Let $X$ be a rigid analytic  space, which is Stein and smooth over $C$. 
Let $i\geq 1$.
\begin{enumerate}
\item  There is a commutative diagram: 
\[ \xymatrix{\Omega^i(X)(-i) \ar[r]^-{{\rm Exp}(-i)} & H^{i+1}_{\proeet}(X, \Q_p(1)) \\
 H^i_{\proeet}(X, \Bdr^+/t),\ar[u]^{\rm PL}_{\wr}\ar[ur]_{\partial_{\rm BdR}} & }\]
where $\partial_{\rm BdR}$ is the edge map coming from the exact sequence of pro-\'etale sheaves~\eqref{sesB} and ${\rm Exp}(-i)$ is the $(-i)$-Tate twist of the map from~\eqref{pStein}.  
\item We have the exact sequence
$$
0\to \Omega^i(X)^{d=0}(-i)\to H^i_{\proeet}(X, \Bdr^+/t)\lomapr{\partial_{\rm BdR}} H^{i+1}_{\proeet}(X, \Q_p(1)).
$$
\end{enumerate}
\end{proposition} 

\begin{proof}
The second claim follows immediately from the first one and diagram \eqref{pStein}.
 For the first claim, note that the exact sequences \eqref{sesB} fit into a commutative diagram 
$$
\xymatrix{
0\ar[r] &\Q_p\ar[r] &  {\mathbb B}_e\ar[r] & ({\mathbb B}_{\dr}/{\mathbb B}^+_{\dr})\ar[r] & 0\\
0\ar[r] &\Q_p(1)\ar[r]\ar[u]^{\wr}_{\rm Id} &  {\mathbb B}^{\phi=p}\ar[r]\ar[u]^{t^{-1}}  & {\mathbb B}^+_{\dr}/t\ar[r]\ar[u]^{t^{-1}}  & 0.
}
$$
This yields that the outer square in the following  diagram commutes.
$$
\xymatrix{
H^i_{\proeet}(X, ({\mathbb B}_{\dr}/{\mathbb B}^+_{\dr})(1))\ar[rr]^{\partial_{}} \ar[dr]^{\pi}& &H^{i+1}_{\proeet}(X,\Q_p(1))\\
& (\Omega^{i}(X)/{\rm Ker}\,d)(-i)\ar@{^(->}[ru]^{{\rm Exp}(-i)}\\
& \Omega^{i}(X)(-i)\ar@{->>}[u]^{\can}\ar[ruu]_{{\rm Exp}(-i)}\\
H^i_{\proeet}(X, {\mathbb B}^+_{\dr}/t)\ar[ru]^{\rm PL}_{\sim}\ar[rr]^{\partial_{\rm BdR}}\ar[uuu]^{t^{-1}}  && H^{i+1}_{\proeet}(X,\Q_p(1)).\ar[uuu]_{\wr}^{\rm Id} 
}
$$
The map $\pi$ is the one from \eqref{sleepy3} and Remark \ref{overC1}. The top triangle commutes by the construction of the map ${\rm Exp}$ described in Section \ref{sleepy2}. Using the computations in Section \ref{bStein}, it is easy to check that the  left trapezoid commutes. This gives us claim (1) of the proposition. 
\end{proof}

\section{The image of ${\rm HTlog}$}

The goal of this section is to prove the following result: 

\begin{theorem}
\label{image00}
Let $X$ be a smooth Stein rigid space over $C$. 
For $i\ge 1$, 
the image of the restriction of the Hodge-Tate logarithm to the group of principal units 
\[ {\rm HTlog}_U : H^i_v(X,U) \to \Omega^i(X)(-i) \]
fits into a short exact sequence of solid $\Z_p$-modules
$$
0 \to  \Omega^i(X)^{d=0}(-i) \to {\rm Im}({\rm HTlog}_U) \xrightarrow{ {\rm Exp}}  \mathcal{I}^i(X) \to 0 
$$
where  $\mathcal{I}^i(X) \subset H^{i+1}_{\proeet}(X, \Q_p(1))$ is the intersection
$${\rm Im}({\rm Exp}) \cap {\rm Im}(\iota^{i+1})={\rm Im}({\rm Exp}) \cap {\rm Ker}(\pi^{i+1}),
$$
where  ${\rm Exp} : (\Omega^i(X)/{\rm Ker}\,d)(-i) \hookrightarrow H^{i+1}_{\proeet}(X, \Q_p(1))$ is the Bloch-Kato exponential map from the previous section and 
$$\iota^j: H^{j}_{\proeet}(X, \Z_p(1)) \to H^j_{\proeet}(X, \Q_p(1)),\quad \pi^j: H^j_{\proeet}(X, \Q_p(1)) \to H^j_{\proeet}(X,  \Q_p/\Z_p(1))
$$
are the canonical morphisms. 
\end{theorem}

\begin{remark}
Alternatively, using diagram \eqref{pStein}, the group 
$ \mathcal{I}^i(X) $ can be seen as exact forms in $\Omega^{i+1}(X)$ coming from $ H^{i+1}_{\proeet}(X, \Z_p(1))$. 
\end{remark}
In particular, we have the following immediate corollary: 

\begin{corollary}
\label{image1}
Let $X$ be a smooth Stein space over $C$. Then, 
\begin{enumerate}
\item The image by ${\rm HTlog}$ of ${\rm Pic}_v(X)$ contains all the closed differentials. More generally, the image by  ${\rm HTlog}$ of $H^i_v(X,\so^{\times})$, $i\geq 1$,   contains all the closed differentials.
\item If the map $H^1_v(X,U) \to {\rm Pic}_v(X)$ is surjective then there is an exact sequence of solid $\Z_p$-modules
$$
0 \to  \Omega^1(X)^{d=0}(-1) \to {\rm Im}({\rm HTlog}) \to \mathcal{I}^1(X) \to 0. 
$$
\end{enumerate}
\end{corollary}

Concerning the first claim of the corollary, note that Heuer already proved in \cite[Cor.\,4.4]{Heu21} that for any smooth rigid space, the image by ${\rm HTlog}$ of ${\rm Pic}_v(X)$ contains all the $df \in \Omega^1(X)$, for $f$ in $\so(X)$. 

\begin{remark}\label{uwagi}
\begin{enumerate}
\item For the affine space ${\bb A}^n_C$, the cohomology group $H^2({\bb A}_C^n, \Z_p)$ is trivial. 
This is why the extra term $\mathcal{I}^1(X)$ does not appear in Heuer's computation. In fact this holds in any degree and we have 
$$
 \Omega^i({\mathbb A}^n_C)^{d=0}(-i) \stackrel{\sim}{\to} {\rm Im}({\rm HTlog}_U). 
$$
\item 
As for smooth Stein curves, it follows from diagram~\eqref{pStein} that the rational cohomology $H^i_{\proeet}(X, \Q_p)$, $i\geq 2$,  is zero, hence $\mathcal{I}^1(X)$ is also trivial in that case. Moreover, $ \Omega^1(X)^{d=0} \stackrel{\sim}{\to}\Omega^1(X)$.
\item Let $i\geq 1$. 
Since we have an exact sequence: 
\[ 
0 \to \Omega^i(X)/{\rm Ker}\,d \xrightarrow{{\rm Exp}} H^{i+1}_{\proeet}(X, \Q_p(i+1)) \to  (H^{i+1}_{{\rm HK}}(X){\otimes}^{\Box}_{\breve{F}}\wh{\B}^+_{\st})^{N=0,\phi=p^{i+1}} \to 0 
\]
we see that the intersection $\mathcal{I}^i(X)$ is zero when the map from $H^{i+1}_{\proeet}(X, \Z_p(i+1))/{T}$, 
where $T$ is the maximal torsion subgroup,  to the Hyodo-Kato term above is injective.  
This will be the case when $X$ is a torus or, more generally,  the analytification of an algebraic variety,  or the Drinfeld upper half space (see below).  

In fact, in general, if we define $\wt{{\mathcal I}}^i(X)(i)$ 
as the kernel of map
\[ H^{i+1}_{\eet}(X, \Z_p(i+1))/T \to (H^{i+1}_{{\rm HK}}(X){\otimes}^{\Box}_{\breve{F}}\wh{\B}^+_{\st})^{N=0,\phi=p^{i+1}}\]
then we have a commutative diagram
 $$
 \xymatrix{  0 \ar[r] &\wt{{\mathcal I}}^i(X)(i)\ar[r] \ar[d]& H^{i+1}_{\eet}(X, \Z_p(i+1))/T \ar[r] \ar[d]^{\can} & (H^{i+1}_{{\rm HK}}(X){\otimes}^{\Box}_{\breve{F}}\wh{\B}^+_{\st})^{N=0,\phi=p^{i+1}}\ar@{=}[d]\\
 0 \ar[r] &  \Omega^i(X)/{\rm Ker}\,d \ar[r]^-{{\rm Exp}} \ar[r]  & H^{i+1}_{\proeet}(X, \Q_p(i+1)) \ar[r] & (H^{i+1}_{{\rm HK}}(X){\otimes}^{\Box}_{\breve{F}}\wh{\B}^+_{\st})^{N=0,\phi=p^{i+1}} \ar[r] &  0,
 }
 $$
where the left square is cartesian. 
 The group $\wt{{\mathcal I}}^i(X)$ surjects onto ${{\mathcal I}}^i(X)$. 
We like to think of the Hyodo-Kato term as carrying $\ell$-adic information, for $\ell\neq p$.
Then  $\wt{{\mathcal I}}^i(X)$ can be seen as
a genuinely $p$-adic phenomena. 
\item 
The groups $\mathcal{I}^j(X)$ need not  be zero in general. For example, they are non-zero for  open unit discs  $D^n_C$ of dimension $n>1$ over $C$. See Section \ref{opendisc} below. 
\end{enumerate}
\end{remark}

\begin{proof}[Proof of Theorem~\ref{image00}]
Recall from   Section~\ref{prelim}, that we can pass from the $v$-topology to the pro-\'etale one without changing the groups $H^i_{?}(X, U)$ and $H^i_{?}(X, \so^{\times})$. We then work on the pro-\'etale site.  
As in \cite[Sec.\,6.2]{Heu21}, we start from the logarithmic exact sequence (point (1) of Lemma~\ref{log}) on the pro-\'etale site: 
\[ 0 \to \Q_p/\Z_p(1) \to U \xrightarrow{\rm log}  \so \to 0.\]
It induces a commutative diagram: 
\begin{equation}\label{ven1}
 \xymatrix{ H^i_{\proeet}(X, U) \ar[r]^-{\rm log} \ar[rd]_{{\rm HTlog}_U} & H^i_{\proeet}(X, \so) \ar[r]^-{\partial_{\rm log}} \ar[r] \ar[d]^-{\rm HT}_{\wr} & H^{i+1}_{\proeet}(X, \Q_p/\Z_p(1)) \\
& \Omega^i(X)(-i), \ar[ru]_{\partial_{\rm log} \circ {\rm HT}^{-1}} &      }
\end{equation}
where the first row is exact. We have used here that $X$ is Stein (the isomorphism ${\rm HT}$ is the one from Remark~\ref{HTstein}). We deduce from the diagram that the image of ${\rm HTlog}_U$ is equal to the kernel of  $\partial_{\rm log} \circ {\rm HT}^{-1}$. We prove in Lemma~\ref{vendredi2} below that the following square commutes: 
\begin{equation}
\label{vendredi}
\xymatrix{ \Omega^i(X)(-i) \ar@{->>}[d] \ar[r]^-{\partial_{\rm log} \circ {\rm HT}^{-1}} & H^{i+1}_{\proeet}(X,\Q_p/\Z_p(1)) \\  
\Omega^i(X)/{\rm Ker}\,d(-i)\ar@{^(->}[r]^{{\rm Exp}} & H^{i+1}_{\proeet}(X, \Q_p(1)). \ar[u]  }
\end{equation} 
It immediately follows that we have an inclusion $\Omega^i(X)^{d=0}(-i) \subset {\rm Ker}(\partial_{\rm log} \circ {\rm HT}^{-1})$. Now, since the kernel of the right vertical map is given by the image of $\iota^{i+1}:  H^{i+1}_{\proeet}(X, \Z_p(1)) \to H^{i+1}_{\proeet}(X, \Q_p(1))$, we get an exact sequence:
\[ 0 \to  \Omega^i(X)^{d=0}(-i) \to {\rm Ker}(\partial_{\rm log} \circ {\rm HT}^{-1}) \xrightarrow{ {\rm Exp}}  {\rm Im}({\rm Exp}) \cap {\rm Im}(\iota^{i+1}) \to 0\]
and this concludes the proof of Theorem~\ref{image00}.
\end{proof}

\begin{lemma}
\label{vendredi2}
The diagram \eqref{vendredi} is commutative. 
\end{lemma}

\begin{proof} 
It suffices to show that we have the commutative  diagram 
\[ \xymatrix{
\Omega^i(X)(-i)  \ar@{=}[d]& \ar[l]_{\rm HT}^{\sim} H^i_{\proeet}(X, \so) \ar[r]^-{\partial_{\rm log}} & H^{i+1}_{\proeet}(X, \Q_p/\Z_p(1)) \\
\Omega^i(X)(-i) \ar@{->>}[d] & \ar[l]_{\rm PL}^{\sim} H^i_{\proeet}(X, \mathbb{B}_{\dr}^+/t)  \ar[u]^{\theta}_{\wr} \ar[r]^{\partial_{\rm BdR}} & H^{i+1}_{\proeet}(X, \Q_p(1)) \ar[u]^{\pi^{i+1}}\\
\Omega^i(X)/{\rm Ker}\,d(-i) \ar[rru]_{\rm Exp} } \]
as the outer trapezoid is exactly the diagram in question. 
The commutativity  of the upper left square is given by Lemma~\ref{sniad3}(\ref{sniad3-2}). The right square comes from the map of short exact sequences: 
\[ \xymatrix{
0 \ar[r] &\Q_p/\Z_p(1) \ar[r] & U \ar[r]^{\rm log} & \so \ar[r] & 0 \\
0 \ar[r] & \Q_p(1) \ar[r] \ar[u]   & \mathbb{B}^{\varphi=p} \ar[r] \ar[u]  & \mathbb{B}_{\rm dR}^+/t\mathbb{B}_{\rm dR}^+ \ar[r] \ar[u]^{\rotatebox{90}{$\sim$}}_{\theta} & 0,}
\]
where the map in the middle is the sharp map appearing in~\eqref{theta-log} (by~\cite[Example~2.22]{LB18}, it makes the right square commutative).    
Hence, it is commutative. The triangle commutes by  Proposition~\ref{ajin}. 
Therefore the outer trapezoid is commutative as well, which is what we wanted.  
\end{proof}

\begin{remark}[Compatibility with Chern classes] \label{chern0}
Let $X$ be a smooth Stein space over $C$. Recall that the pro-\'etale  Chern class map
\begin{equation}
\label{chern1}
 c_1^{\proeet} : H^1_{\proeet}(X, U) \to H^2_{\proeet}(X, \Z_p(1)) 
 \end{equation}
 can be defined 
as the edge morphism coming from the short exact sequence (on $X_{\proeet}$) 
\[0 \to \Z_p(1) \to \lim_{x \mapsto x^p} U \to U \to 0 \]
(which is obtained, by passing to the limit, from  the Kummer exact sequence :
\[ 0 \to \mu_{p^n} \to U \xrightarrow{p^n} U \to 0.) \]
Similarly, we have the  \'etale Chern class map {(obtained by taking the inverse limit of the boundary maps of the Kummer sequences on the \'etale site)}
\[ c_1^{\eet} : H^1_{\eet}(X, U) \to H^2_{\eet}(X, \Z_p(1)).  \]
It is compatible with~\eqref{chern1} via the canonical map $H^2_{\eet}(X, \Z_p(1)) \stackrel{\sim}{\to} H^2_{\proeet}(X, \Z_p(1)) $.

Now,  we can identify the map ${\bb B}^{\varphi=p}\xrightarrow{\theta} \so$ with the map $\lim_{x \mapsto x^p} U \to U  \xrightarrow{\log} \so$ (see~\eqref{theta-log}) and we have a commutative diagram with exact rows (see the proof of Lemma \ref{vendredi2})
\[ \xymatrix{ 0 \ar[r] &  \Q_p(1) \ar[r]&  {\bb B}^{\varphi=p} \ar[r]^{\theta} & \so \ar[r] &  0 \\
  0 \ar[r] & \Z_p(1) \ar[u] \ar[r] & \lim_{x \mapsto x^p} U  \ar[r] \ar[u] & U \ar[u]_{\log} \ar[r]  & 0 } \] 
 Passing to cohomology groups we obtain a commutative diagram
 $$
 \xymatrix{
 H^1_{\proeet}(X,U)\ar[r] ^{c^{\proeet}_1}\ar[d]^{\log}& H^2_{\proeet}(X,\Z_p(1))\ar[d]\\
 H^1_{\proeet}(X,\so)\ar[r]^-{\partial} & H^2_{\proeet}(X,\Q_p(1)).
 }
 $$
 Putting it all together and   using the compatibility between the Hodge-Tate morphism and the Bloch-Kato exponential from Lemma \ref{sniad3}(\ref{sniad3-2}) and Proposition \ref{ajin}, we obtain a commutative diagram: 
 \begin{equation}
 \label{chern2} \xymatrix{ 
 H^1_{\eet}(X, U) \ar@{^(->}[r] \ar[d]_{c_1^{\eet}}& H^1_{\proeet}(X,U)\ar[r]^-{{\rm HTlog}_U} \ar[d]_{c_1^{\proeet}} & \Omega^1(X)(-1)/\kker\,d \ar@{^(->}[d]^{\rm Exp} \\
 H^2_{\eet}(X, \Z_p(1)) \ar[r]^-{\sim} & H^2_{\proeet}(X, \Z_p(1)) \ar[r] & H^2_{\proeet}(X, \Q_p(1)). }
 \end{equation}
 We note that it implies that the image of the rational pro-\'etale Chern class  map ${c_1^{\proeet}}: H^1_{\eet}(X, U)\to H^2_{\proeet}(X, \Q_p(1))$ is always zero. Moreover, we get an exact sequence
 $$
 0\to  H^1_{\eet}(X, U) \to  H^1_{\proeet}(X,U)_0\lomapr{{\rm HTlog}_U} \Omega^1(X)^{d=0} (-1)\to 0 ,
 $$
 where we set $H^1_{\proeet}(X,U)_0:=\kker({c_1^{\proeet}}: H^1_{\proeet}(X,U)\to H^2_{\proeet}(X, \Q_p(1)))$, the subgroup of classes (rationally) homologically trivial. 
 \end{remark}
\begin{remark}
The previous remark  brings us to a question: if we define ${\rm Pic}_v(X)_0$ similarly to $H^1_{\proeet}(X,U)_0$ replacing $U$ by $\mathbb{G}_m$,  
and set ${\rm Pic}_{\rm an}(X)_0:={\rm Pic}_{\rm an}(X)\cap {\rm Pic}_{v}(X)_0$, 
 do we have an exact sequence
$$
 0\to  {\rm Pic}_{\rm an}(X)_0\to    {\rm Pic}_v(X)_0\lomapr{{\rm HTlog}} \Omega^1(X)^{d=0} (-1)\to 0 \quad ?
$$
  To answer this we need to better understand the canonical maps $H^1_{\eet}(X, U)\to {\rm Pic}_{\rm an}(X)$ and $H^1_{\proeet}(X,U)\to {\rm Pic}_v(X)$ for Stein spaces.
\end{remark}

\section{Examples}\label{examples}

Let us look now at some examples of computations of $H^i_{\tau}(X,{\mathbb G}_m)$, for $i=1,2$, 
for certain smooth Stein space $X$ over $C$.

\subsection{Picard group} We start with the Picard group. 

\subsubsection{Curves} 

Smooth Stein rigid analytic varieties $X$ of dimension $1$ were already treated in \cite[Sec.\,4.1]{Heu21}: 
as $H^2_{\eet}(X,\so^{\times})$ vanishes in this case, 
the exact sequence \eqref{even3} becomes 
\begin{equation}\label{nie0}
0\to {\rm Pic}_{\rm an}(X)\to {\rm Pic}_v(X)\lomapr{\rm HTlog} \Omega^1(X)\to 0.
\end{equation}
Moreover, if such a  curve  $X$ is defined over $K$, we have $H^2_{\proeet}(X,\Q_p)=0$ (because $H^2_{\dr}(X)=0$) hence ${\mathcal I}^1(X)=0$. 
 
\subsubsection{Affine space} \label{nie1}
The case of the affine space was treated in  two different ways in \cite[Sec.\,6]{Heu21}. 
Our approach here  is similar to the one presented in \cite[Sec.\,6.2]{Heu21}. 
Let ${\bb A}^n_{C}$ be the rigid analytic affine space of dimension $n$ over $C$. For $i\geq 1$, since $H^i_{\dr}({\mathbb A}^n_C)=0$ and, hence, $H^i_{\hk}({\mathbb A}^n_C)=0$, by diagram \eqref{pStein}, we have an isomorphism
$$
 \Omega^i({\bb A}^n_{C})/{\rm Ker}\,d \stackrel{\sim}{\to} H^{i+1}_{\proeet}({\bb A}^n_{C}, \Q_p(i+1)).
$$
Since $H^{i+1}_{\proeet}({\bb A}^n_{C}, \Z_p(i+1))=0$ (by comparison with the algebraic case),  we have ${\mathcal I}^i({\bb A}^n_{C})=0$. Thus, by Theorem \ref{image00}, we have an isomorphism
\begin{equation}\label{nie12}
{\rm Im}({\rm HTlog_{U,i}}: H^i_v({\bb A}^n_{C},U)\to  \Omega^i({\bb A}^n_{C})(-i))\stackrel{\sim}{\leftarrow} \Omega^i({\bb A}^n_{C})^{d=0}(-i).
\end{equation}
Moreover, by \cite[Lem.\,6.7]{Heu21},   the map from $H^1_v({\bb A}^n_{C},U)$ to the $v$-Picard group is surjective. 
Thus we obtain an exact sequence
\[ 0 \to {\rm Pic}_{\rm an}({\bb A}^n_{C}) \to {\rm Pic}_{v}({\bb A}^n_{C}) \to \Omega^1({\bb A}^n_{C})^{d=0}(-1) \to 0. \]
Since the analytic Picard group of the affine space is trivial\footnote{In fact, by \cite[Ch.\,V.3, Prop.\,2]{Gru68}, every analytic vector bundle on ${\bb A}^d_{C}$ is trivial.}, this implies that the Hodge-Tate logarithm is an isomorphism
$$
{\rm HTlog}:\quad  {\rm Pic}_{v}({\bb A}^n_{C})\stackrel{\sim}{\to} \Omega^1({\bb A}^n_{C})^{d=0}(-1).
$$

\subsubsection{Torus}\label{nie2}

Consider  the rigid analytic torus ${\bb G}^n_{m,C}$ of dimension $n$ over $C$. This case is similar to the case of affine space because the analytic Picard group is trivial (see \cite[Th.\,6.3.3]{FVdP}) but also different because the de Rham cohomology is non-trivial (though of finite rank).  

  For $i\geq 0$, by \eqref{pStein}, we have the exact sequence (see~\cite[Sec.\,4.3.2]{CDN3}): 
\begin{equation}\label{tea1} 0 \to \Omega^i({\bb G}^n_{m,C})/{\rm Ker}\,d \to H^{i+1}_{\proeet}({\bb G}^n_{m,C}, \Q_p(i+1)) \to \wedge^{i+1} \Q^n_p \to 0. 
\end{equation}
 Since we have $H^{i+1}_{\eet}({\bb G}^n_{m,C}, \Z_p(i+1))\simeq  \wedge^{i+1}\Z^n_p$ (compare with the \'etale cohomology of the algebraic torus), we see that the map from the integral cohomology to the Hyodo-Kato term is injective. We used here that the projection from pro-\'etale cohomology to the Hyodo-Kato term is compatible with products and symbol maps: this is because the comparison theorem between the pro-\'etale cohomology and syntomic cohomology and the projection from syntomic cohomology to the Hyodo-Kato term both satisfy these compatibilities. 
We obtain that the intersection between the elements coming from $\Omega^1({\bb G}^n_{m,C})/{\rm Ker}\,d$ and the ones coming from the integral pro-\'etale cohomology is trivial. Thus, by Theorem \ref{image00}, we have an isomorphism
\begin{equation}\label{tea2}
{\rm Im}({\rm HTlog_{U,i}}: H^i_v({\bb G}^n_{m,C},U)\to  \Omega^i({\bb G}^n_{m,C})(-i))\stackrel{\sim}{\leftarrow} \Omega^i({\bb G}^n_{m,C})^{d=0}(-i).
\end{equation}

Moreover, we have: 
\begin{lemma}
\label{surject-U}
The map from $H^1_v({\bb G}^n_{m,C},U)$ to $H^1_v({\bb G}^n_{m,C}, \so^{\times})$ is surjective.
\end{lemma}

\begin{proof}
We will show that $H^1_v({\bb G}_{m,C}^n,\overline{\so}^{\times})$ is zero. 
Let  $\{ X_k \}_{k \in \N}$ be  the  Stein covering of ${\bb G}_{m,C}^n$  from \cite[Proof of Th.\,7.1]{Jun21}.    
On each $X_k$ the sheaf  $\so^{\times}_{\rm an}$ is acyclic. 
We have the exact sequence
\begin{equation}\label{kawa11}
0\to \R^1\lim_k H^0_v(X_k,\overline{\so}^{\times})\to H^1_v({\bb G}_{m,C}^n,\overline{\so}^{\times})\to \lim_kH^1_v(X_k,\overline{\so}^{\times})\to 0
\end{equation}
But, by   \cite[Lem.\,2.14]{Heu2}, \cite[Proof of Th.\,7.1]{Jun21}
we have 
\begin{equation}\label{drinfeld10}
H^0_v(X_k,\overline{\so}^{\times})=H^0_{\rm an}(X_k,\overline{\so}^{\times})[1/p]=M_k[1/p],
\end{equation}
where $M_k$ is an abelian group of finite type. Moreover, the maps $M_{k+1}\to M_k$ are surjective, 
and they remain so after inverting $p$. 
We get 
$\R^1\lim_k H^0_v(X_k,\overline{\so}^{\times})=0 $.

We also claim that  $H^1_v(X_k, \overline{\so}^{\times})=0$, for all $k$. 
Indeed, by the point (2) of Proposition~\ref{topo}, we see that it suffices to check this for the \'etale topology. Using the exponential sequence (point (1) of Lemma~\eqref{log}), it is enough to show that 
\[ H^1_{\eet}(X_k, \so^{\times}[\frac{1}{p}])= 0 \text{ and } H_{\eet}^2(X_k, \so)=0. \]
The second equality follows from the fact that $X_k$ is an affinoid. For the first one, 
we use that ${\rm Pic}_{\eet}(X_k)= {\rm Pic}_{\rm an}(X_k)=0$ and since $X_k$ is quasi-compact, we have $H^1_{\eet}(X_k, \so^{\times}[\frac{1}{p}])=H^1_{\eet}(X_k, \so^{\times})[\frac{1}{p}]=0$.  
Our claim   follows. 

Hence
$ \lim_kH^1_v(X_k,\overline{\so}^{\times})=0,
$
and, by \eqref{kawa11}, $H^1_v({\bb G}_{m,C}^n,\overline{\so}^{\times})=0$, as wanted.
\end{proof}
 
Thus we obtain an exact sequence
\[ 0 \to {\rm Pic}_{\rm an}({\bb G}^n_{m,C}) \to {\rm Pic}_{v}({\bb G}^n_{m,C}) \to \Omega^1({\bb G}^n_{m,C})^{d=0}(-1) \to 0. \]
Since the analytic Picard group of the torus is trivial, this means that the Hodge-Tate logarithm is an isomorphism
\begin{equation}
\label{torus}
{\rm HTlog}: \quad {\rm Pic}_{v}({\bb G}^n_{m,C})\stackrel{\sim}{\to}\Omega^1({\bb G}^n_{m,C})^{d=0}(-1).
\end{equation}
\begin{remark}
Let $n_1,n_2\geq 0$. Combining \eqref{torus} and \cite[Th.\,6.1]{Heu21}, we obtain an isomorphism
\begin{equation}\label{tea3}
{\rm HTlog}:\quad {\rm Pic}_{v}({\bb G}^{n_1}_{m,C}\times{\mathbb A}^{n_2}_C)\stackrel{\sim}{\to}\Omega^1({\bb G}^{n_1}_{m,C}\times{\mathbb A}^{n_2}_C)^{d=0}(-1).
\end{equation}
This isomorphism can be also obtained arguing as above,  in the case of the torus. 
More precisely, since de Rham cohomology satisfies a K\"unneth formula, 
replacing ${\bb G}^{n}_{m,C}$ with ${\bb G}^{n_1}_{m,C}\times{\mathbb A}^{n_2}_C$ yields an analogue of the exact sequence \eqref{tea1} 
and then also an analogue of 
the isomorphism \eqref{tea2}. 
The rest of the argument goes through yielding \eqref{tea3}.
\end{remark}
\begin{remark}
The following two examples use Hyodo-Kato GAGA and an almost proper $C_{\st}$-comparison theorem which are proved in the forthcoming PhD thesis (from Sorbonne University) of Xinyu Shao. The semistable reduction case of the latter was shown 
in \cite[Cor.\,1.10]{CN1}. 
\end{remark}
\subsubsection{Analytification of algebraic varieties} 

The examples of the affine space and the torus generalise. Let $X$ be the analytification of an affine  smooth algebraic variety $X^{\rm alg}$ over $C$. 
Then $X$ is Stein.  
In this case both the algebraic and the analytic de Rham cohomologies are of finite rank (they are functorially  isomorphic but not as filtered objects). 
Let $i\geq 1$. 
We have a commutative diagram (see Remark \ref{uwagi})
 $$
 \xymatrix{ &0\ar[r] & H^{i+1}_{\eet}(X^{\rm alg}, \Z_p(i+1))/T \ar[r] \ar[d]^{\can}_{\wr} & (H^{i+1}_{{\rm HK}}(X^{\rm alg}){\otimes}^{\Box}_{\breve{F}}\wh{\B}^+_{\st})^{N=0,\phi=p^{i+1}}\ar[d]^{\can}_{\wr}\\
  0 \ar[r] &\wt{{\mathcal I}}^i(X)(i)\ar[r] \ar[d]& H^{i+1}_{\eet}(X, \Z_p(i+1))/T \ar[r] \ar[d]^{\can} & (H^{i+1}_{{\rm HK}}(X){\otimes}^{\Box}_{\breve{F}}\wh{\B}^+_{\st})^{N=0,\phi=p^{i+1}}\ar@{=}[d]\\
 0 \ar[r] &  \Omega^i(X)/{\rm Ker}\,d \ar[r]^-{{\rm Exp}} \ar[r]  & H^{i+1}_{\proeet}(X, \Q_p(i+1)) \ar[r] & (H^{i+1}_{{\rm HK}}(X){\otimes}^{\Box}_{\breve{F}}\wh{\B}^+_{\st})^{N=0,\phi=p^{i+1}} \ar[r] &  0.
 }
 $$
The rows are exact. 
For the top row this follows from the algebraic $p$-adic comparison theorems \cite[Sec.\,3.2 and Sec.\,3.3]{Bei1}. 
The top square commutes by the compatibility of the algebraic and analytic $p$-adic comparison morphisms. 
This fact and the proof of the isomorphism  between the algebraic and the analytic Hyodo-Kato cohomologies can be found  in \cite{Shao}. 
It follows that 
 $\wt{{\mathcal I}}^i(X)=0$ and hence ${{\mathcal I}}^i(X)=0$. We have proved:
 \begin{corollary}\label{morn12}Let $X$ be the analytification of an affine  smooth algebraic variety over $C$.  Let $i\geq 1$. Then 
${{\mathcal I}}^i(X)=0$ and we have an isomorphism
$$
{\rm Im}({\rm HTlog_{U,i}}: H^i_v(X,U)\to  \Omega^i(X)(-i))\stackrel{\sim}{\leftarrow} \Omega^i(X)^{d=0}(-i).
$$
 \end{corollary}
 
 \subsubsection{Almost proper varieties} 
 Even more generally, let $X$ be a smooth rigid analytic variety over $C$ that is of the form  $X=Y\setminus Z$, where $Y$ is a proper and smooth rigid analytic variety over $C$ and $Z$ is a closed rigid analytic subvariety of $Y$. Assume that $X$ is Stein. 
 Let $i\geq 1$. We have a commutative diagram
 \begin{equation}\label{pies1}
 \xymatrix{ 
  0 \ar[r] &\wt{{\mathcal I}}^i(X)(i)\ar[r] \ar[d]& H^{i+1}_{\eet}(X, \Z_p(i+1))/T \ar[r]^-{\alpha} \ar[d]^{\can} & (H^{i+1}_{{\rm HK}}(X){\otimes}^{\Box}_{\breve{F}}\wh{\B}^+_{\st})^{N=0,\phi=p^{i+1}}\ar@{=}[d]\\
 0 \ar[r] &  \Omega^i(X)/{\rm Ker}\,d \ar[r]^-{{\rm Exp}} \ar[r]  & H^{i+1}_{\proeet}(X, \Q_p(i+1)) \ar[r] & (H^{i+1}_{{\rm HK}}(X){\otimes}^{\Box}_{\breve{F}}\wh{\B}^+_{\st})^{N=0,\phi=p^{i+1}} \ar[r] &  0.
 }
 \end{equation}
  The rows are exact.  The map $\alpha$  fits into the following commutative diagram
 $$
 \xymatrix{H^{i+1}_{\eet}(X, \Z_p(i+1))/T\ar[rd]^-{\alpha}\ar@{^(->}[d] \\
 H^{i+1}_{\eet}(X, \Q_p(i+1)) \ar@{^(->}[r]^-{\alpha_{\Q_p}}  & (H^{i+1}_{{\rm HK}}(X){\otimes}^{\Box}_{\breve{F}}\wh{\B}^+_{\st})^{N=0,\phi=p^{i+1}},}
 $$
 where the map $\alpha_{\Q_p}$ comes from    the standard (almost proper)  rational $C_{\rm pst}$ comparison theorem and we know that it is injective (proved in \cite{Shao}). The vertical map is injective because $ H^{i+1}_{\eet}(X, \Q_p(i+1))\simeq
  H^{i+1}_{\eet}(X, \Z_p(i+1))\otimes_{\Z_p}\Q_p$. It follows that the map $\alpha$ is injective as well. 
Hence, diagram  \eqref{pies1} implies
 that 
 $\wt{{\mathcal I}}^i(X)=0$ and, thus,   ${{\mathcal I}}^i(X)=0$. We have proved an analogue of Corollary \ref{morn12} for $X$.
 
\subsubsection{Drinfeld space} \label{nie3}In this example, the analytic Picard group is also trivial (see \cite[Th.\,A]{Jun21})  but the de Rham cohomology is not of finite rank anymore.
For $n$ an integer, the Drinfeld space over $K$ of dimension $n$ is defined by \[ {\bb H}^n_K:={\bb P}^n_K\setminus \bigcup_{H \in \mathcal{H} } H \] where $\mathcal{H}:={\bb P}((K^{n+1})^{\times})$ denotes the set of the $K$-rational hyperplanes in the rigid-analytic projective space ${\bb P}^n_K$.

For $\Lambda$ a topological ring, we denote by ${\rm Sp}_r(\Lambda)$ the associated generalised Steinberg representation and by ${\rm Sp}_r(\Lambda)^*$ its dual. Recall that we have the following computation:

\begin{proposition}\cite[Th.\,1.3]{CDN3}\cite[Th.\,1.1]{CDN4} 
\label{drinfeld1}
Let $i \ge 0$. Then:
\begin{enumerate}
\item There is an exact  sequence: 
\[ 0 \to \Omega^{i-1}({\bb H}_C^n)/{\rm Ker}\,d \to H^i_{\proeet}({\bb H}_C^n, \Q_p(i)) \to {\rm Sp}_i(\Q_p)^* \to 0. \]
\item There are  isomorphisms:  
\[ H^i_{\eet}({\bb H}_C^n, \Z_p(i)) \simeq {\rm Sp}_i(\Z_p)^* \quad \text{ and } \quad H^i_{\eet}({\bb H}_C^n, \Q_p(i)) \simeq {\rm Sp}^{\rm cont}_i(\Q_p)^*. \]
\item The above morphisms are compatible, i.e. there is a commutative diagram 
\[ \xymatrix{ H^i_{\eet}({\bb H}_C^n, \Z_p(i))\otimes_{\Z_p} \Q_p \ar[r]^-{\sim} \ar[d]_{\rotatebox{90}{$\sim$}} &H^i_{\eet}({\bb H}_C^n, \Q_p(i)) \ar[d]_{\rotatebox{90}{$\sim$}} \ar[r]  & H^i_{\proeet}({\bb H}_C^n, \Q_p(i)) \ar@{->>}[d] \\
{\rm Sp}_i(\Z_p)^*\otimes_{\Z_p} \Q_p  \ar[r]^-{\sim} &{\rm Sp}^{\rm cont}_i(\Q_p)^* \ar@{^(->}[r] & {\rm Sp}_i(\Q_p)^*.} \]
\end{enumerate}
\end{proposition} 

As in the case of torus, we see that the map from $H^i_{\proeet}({\bb H}_C^n, \Z_p(i))$ to the Hyodo-Kato term is injective\footnote{The same argument concerning compatibilities applies.}, and we deduce that the intersection $\mathcal{I}^i({\bb H}_C^n)$ is zero. Moreover, we have: 

\begin{lemma}
\label{surject1-U}
The map from $H^1_v({\bb H}_C^n,U)$ to $H^1_v({\bb H}_C^n, \so^{\times})$ is surjective.
\end{lemma}

\begin{proof}
Analogous to the proof of Lemma \ref{surject-U}. 
\end{proof}

 Using Lemma \ref{surject1-U}, 
 we get an exact sequence: 
\[ 0 \to {\rm Pic}_{\rm an}({\bb H}_C^n) \to {\rm Pic}_{v}({\bb H}_C^n) \to \Omega^1({\bb H}_C^n)^{d=0}(-1) \to 0. \]
Since the analytic Picard group of the Drinfeld space is trivial, finally, we obtain an isomorphism 
\begin{equation}
\label{drinfeld}
{\rm HTlog}:\quad {\rm Pic}_v({\bb H}_C^n) \stackrel{\sim}{\to}\Omega^1({\bb H}_C^n)^{d=0}(-1).
\end{equation}

\subsubsection{Open disc}\label{opendisc}

 Let now $n>1$. Consider  the open unit disc $D^n$ of dimension $n$ over $C$. We will prove that the intersection $\mathcal{I}^1(D^n)$ is nonzero, which shows that the image of the Hodge-Tate logarithm need not  be reduced to the closed differentials in general. 

  Write $D^n=D_1\times_CD_2$, where $D_1$, $D_2$ are open unit discs of dimension $1$ and $n-1$, respectively. Choose functions $f_i\in\so^*(D_{i})$.  We have $d \log f_i\in \Omega^1(D_{i})$, $\omega:=d \log f_1\wedge d \log f_2\in \Omega^2(D^n)$, and clearly $d\omega=0$, i.e.,  $\omega\in \Omega^2(D^n)^{d=0}$. We note that, since de Rham cohomology of $D^n$ is trivial in positive degrees, we have the isomorphism
$$
d: \Omega^1(D^n)/{\rm Ker}\,d\stackrel{\sim}{\to}\Omega^2(D^n)^{d=0}.
$$
and, from diagram \ref{pStein}, the commutative diagram
  $$
  \xymatrix{
  \Omega^1(D^n)/{\rm Ker}\,d\ar[r]_-{\sim}^-{\rm Exp} \ar@{=}[d] & H^2_{\proeet}(D^n,\Q_p(2))\ar[d]^-{{\rm dLog}}_{\wr}\\
  \Omega^1(D^n)/{\rm Ker}\,d\ar[r]_-{\sim}^{d} & \Omega^2(D^n)^{d=0}.
  }
  $$
  
Let now $\delta(f_1), \delta(f_2)$ be the images by the Kummer maps $\delta: \so^*(D_{i})\to H^1_{\eet}(D_{i},\Z_p(1))$ of $f_1,f_2$. Then $\delta(f_1)\cup \delta(f_2)\in H^2_{\eet}(D^n,\Z_p(2))$. (Here we abuse the notation slightly.) 
  We claim that the image of $\delta(f_1)\cup \delta(f_2)$ in $ \Omega^2(D^n)^{d=0}$ is equal to  $\omega$. Indeed, we compute
  \begin{align}\label{morning1}
  d{\rm Log}(\delta(f_1)\cup\delta(f_2))= d{\rm Log}(\delta(f_1))\cup d{\rm Log}(\delta(f_2))=d\log f_1\cup d \log f_2=\omega.
  \end{align}
  The first equality follows from the fact that the map $d{\rm Log}$ commutes with cup products: it is defined using comparison with syntomic cohomology and the composition
  $$
  H^2_{\proeet}(D^n,\Q_p(2))\to H^2(F^2\R\Gamma_{\dr}(D^n/\B^+_{\dr}))\stackrel{\theta}{\to} \Omega^2(D^n)^{d=0};
  $$
  both of which commute with cup products. The second equality in \eqref{morning1} is the compatibility of the \'etale and de Rham symbol maps: the symbol maps are induced by the first Chern class maps and the passage from \'etale cohomology to syntomic one as well as the projection from the latter to the filtered de Rham cohomology are both compatible with the Chern class maps. 
  
  Now, it suffices to make sure that $\omega\neq 0$. But, for that it is enough to choose nonconstant functions $f_1, f_2$. 
 An analogous argument will show that $\mathcal{I}^i(D^n)$ is nonzero for all $n-1\ge i\ge 1$. 
 
\begin{remark}
 This example is a curious one: de Rham cohomology is trivial in positive degrees but the analytic Picard group depends on the ground field (it will be non-trivial in our case).  More precisely, recall that we have the following result (see \cite[Prop.\,3.5]{Sig17}):
\begin{proposition}  {\rm (\cite[Ch.\,V, Prop.\,2]{Gru68})} Let $L$   be a complete, non-archimedean, non-trivially valued field. Let 
$r \in (|L| \cup \{\infty\})^n$  and let $X_r$ be an open polydisc of
polyradius $r$:
$$X_r =\bigcup_{|\eta_i|<r_i}{\rm Sp}(L<\eta_1^{-1}T_1,\cdots, {\eta_n}^{-1}T_n>).
$$
The Picard group ${\rm Pic}(X_r)$ is trivial if and only if one of the following holds:
\begin{enumerate}
\item  the field $L$ is spherically complete or
\item  the polyradius is $r = (\infty, . . . ,\infty)$, that is, $X_r = {\mathbb A}^n_L$ is the analytic affine space.
\end{enumerate}
\end{proposition}
The ``only if'' part was shown  by Lazard \cite[Prop.\,6]{Laz62}: Assume that  $L$ is not  spherically
complete and let $r \in  |L|$; then Lazard constructs a divisor on $X_r$ which is not a principal divisor. This
implies that open discs $X_r$ which are bounded in at least one direction have non-trivial line bundles
(see \cite[p. 87, Rem.\,2]{Gru68}). 
\end{remark}

 We have: 
\begin{lemma}
The canonical map $H^1_v(D^n,U)\to {\rm Pic}_{v}(D^n) $ is surjective.
\end{lemma}

\begin{proof}
We will show that $H^1_v(D^n,\overline{\so}^{\times})$ is zero. 
Take $\{ X_k \}_{k \in \N}$ -- a  Stein covering of $D^n$ by closed balls $X_k$. We have the exact sequence
\begin{equation}\label{kawa1}
0\to \R^1\lim_k H^0_v(X_k,\overline{\so}^{\times})\to H^1_v(D^n,\overline{\so}^{\times})\to \lim_kH^1_v(X_k,\overline{\so}^{\times})\to 0
\end{equation}
But, by   \cite[Lem.\,{6.6}]{Heu21}), 
\begin{equation}\label{closedball}
H^0_v(X_k,\overline{\so}^{\times})=C^{\times}/(1+\mathfrak{m}\so_C),\quad H^1_v(X_k,\overline{\so}^{\times})=0. 
\end{equation}
Hence
$$
\R^1\lim_k H^0_v(X_k,\overline{\so}^{\times})=0,\quad \lim_kH^1_v(X_k,\overline{\so}^{\times})=0.
$$
Thus, by \eqref{kawa1}, $H^1_v(D^n,\overline{\so}^{\times})=0$, as wanted.
\end{proof}
Hence  we have exact sequences of non-trivial groups: 
\begin{align*}
 0 \to {\rm Pic}_{\rm an}(D^n) \to {\rm Pic}_{v}(D^n) \to {\rm Im}({\rm HTlog}) \to 0,\\
0\to  \Omega^1(D^n)^{d=0}(-1)\to {\rm Im}({\rm HTlog}) \stackrel{\rm Exp} {\longrightarrow}{\mathcal I}^1(D^n)(-1)\to 0.
\end{align*}
We note that  $\Omega^1(D^n)^{d=0}\stackrel{\sim}{\leftarrow} \so(D^n)/C$. 

  Since $H^1_{\rm an}(D^n,U)\stackrel{\sim}{\to}{\rm Pic}_{\rm an}(D^n)$ and $H^1_{v}(D^n,U)\stackrel{\sim}{\to}{\rm Pic}_{v}(D^n)$ (use Proposition \ref{topo}), Remark \ref{chern0}  yields the exact sequence
$$
 0 \to {\rm Pic}_{\rm an}(D^n) \to {\rm Pic}_{v}(D^n)_0 \to  \Omega^1(D^n)^{d=0}(-1) \to 0.
$$
\subsection{The group $H^2_{\eet}(X, \mathbb{G}_m)$} 
The above computations allow us to deduce a little bit about the structure of the groups $H^2_{\eet}(X, \mathbb{G}_m)$. 

 Let $X$ be a smooth Stein space of dimension $n$ over $C$.
 
 $\bullet$ {\em The case of   $n\geq i\geq 2$.}  
By Corollary \ref{sally1},  we have an exact sequence 
\begin{equation}\label{cold1}
 0 \to {\rm Coker}({\rm HTlog}_{i-1}) \to H^i_{\eet}(X, \so^{\times}) \lomapr{\nu_i^*}   {\rm Ker}({\rm HTlog}_{i}) \to 0. 
\end{equation}
We have ${\rm Coker}({\rm HTlog}_{i-1})=\Omega^{i-1}(X)(-i+1)/{\rm Im}({\rm HTlog}_{i-1})$. Since ${\rm Im}({\rm HTlog}_{i-1})\supset {\rm Im}({\rm HTlog}_{U, i-1})$, we have 
${\rm Coker}({\rm HTlog}_{U,i-1})\twoheadrightarrow{\rm Coker}({\rm HTlog}_{i-1})$. In the case that the inclusion map above  is an isomorphism, by Theorem \ref{image00}, we have an exact sequence
$$
0\to \Omega^{i-1}(X)^{d=0}(-i+1)\to {\rm Im}({\rm HTlog}_{i-1})\to {\mathcal I}^{i-1}(X)\to 0,
$$
which,  in combination with the exact sequence  \eqref{cold1},   yields the exact sequence\footnote{The expression  in the denominator denotes an extension of the two terms.}
\begin{equation}\label{cold2}
 0 \to \Omega^{i-1}(X)(-i+1)/[{\rm Ker}\, d- {\mathcal I}^{i-1}(X)]\to H^i_{\eet}(X, \so^{\times}) \lomapr{\nu^*_i}   {\rm Ker}({\rm HTlog}_{i}) \to 0. 
\end{equation}
 
  $\bullet$ {\em The case of  $n=1$.} By Proposition \ref{ses-leray}  we have
  $$
0= H^2_{\eet}(X,\so^{\times})\stackrel{\sim}{\to} H^2_{v}(X,\so^{\times}).
  $$

 Hence, by Sections \ref{nie1}, \ref{nie2}, and \ref{nie3},  if $X$ is an affine space, a torus, or a Drinfeld space (base changed to $C$), we have the exact sequence
 \begin{equation}\label{basic}
  0 \to \Omega^{1}(X)(-1)/{\rm Ker}\, d\to H^2_{\eet}(X, {\mathbb G}_m) \lomapr{\nu^*_2}   {\rm Ker}({\rm HTlog}_{2}) \to 0. 
\end{equation}

 Moreover, in the case $H^1_{\eet}(X,\overline{\so}^{\times})=0$, we have the injection $${\rm Ker}({\rm HTlog}_{U,2})\hookrightarrow {\rm Ker}({\rm HTlog}_{2}).$$
 Again, this is the case for an affine space, a torus, or a Drinfeld space.

\subsubsection{Comparison with $p$-adic cohomology}  

To get a handle on ${\rm Ker}({\rm HTlog}_{U,i})$, we can 
use the logarithmic exact sequence
\[ 0 \to (\Q_p/\Z_p)(1) \to U \xrightarrow{\rm log} \so \to 0 \]
and Lemma~\ref{vendredi2} to obtain the bottom sequence  in the following  commutative diagram with exact rows: 
\begin{equation}
\label{nov}
\xymatrix{
0 \ar[r] &(\Omega^{i-1}(X)/{\rm Ker}\,d)(1-i) \ar[r]^-{\rm Exp} \ar[d]^{f_1} & H^i_{\proeet}(X, \Q_p(1)) \ar[d]^{f_2} \ar[r] & {\rm HK}^i(X)(1) \ar[d]^{f_3} \ar[r] & 0 \\ 
0 \ar[r] &{\rm Coker}({\rm HTlog}_{U,{i-1}})  \ar[r] &  H^i_{\proeet}(X, \Q_p/\Z_p(1)) \ar[r] & {\rm Ker}({\rm HTlog}_{U,i}) \ar[r] & 0,} \end{equation}  
where we set ${\rm HK}^i(X):=(H^i_{\rm HK}(X)\otimes^{\Box}_{\breve{F}}\wh{\B}^+_{\st})^{N=0,\phi=p^i}$. 
The top row comes from diagram \ref{pStein}.  
The left square commutes by diagram \eqref{vendredi}. 
Hence, if  the map $f_1$  in~\eqref{nov} is an isomorphism the right square in~\eqref{nov} is bicartesian. If we also have isomorphisms ${\rm Coker}({\rm HTlog}_{U,i-1}) \stackrel{\sim}{\to} {\rm Coker}({\rm HTlog}_{i-1})$ and $ {\rm Ker}({\rm HTlog}_{U,i})\stackrel{\sim}{\to}{\rm Ker}({\rm HTlog}_i)$, we can completely determine ${\rm Ker}({\rm HTlog}_i)$ using $p$-adic cohomologies.

\subsubsection{Examples of ${\rm Ker}({\rm HTlog}_{U,i})$} We present here the following computation:
\begin{proposition}\label{Bexamples}
For $i\geq 1$, we have the natural isomorphisms:
\begin{align*}
& {\rm Ker}({\rm HTlog}_{U,i}({\mathbb A}^n_C)) =0,\\
& {\rm Ker}({\rm HTlog}_{U,i}({\mathbb G}^n_{m,C})) \cong \wedge^i(\Q_p/\Z_p)^n,\\
&{\rm Ker}({\rm HTlog}_{U,i}({\mathbb H}^n_{C})) \cong {\rm Sp}_i(\Z_p)^{\vee},
\end{align*}
denoting by $(-)^{\vee}$ the Pontryagin dual.
\end{proposition}
\begin{proof}
Let us start with the affine space $X={\mathbb A}^n_C$. Since $H^i_{\proeet}({\bb A}^n_C, \Z_p)=0$, for $i\geq 1$,  the canonical map $H^i_{\proeet}({\bb A}^n_C, \Q_p)\to H^i_{\proeet}({\bb A}^n_C, \Q_p/\Z_p)$ is an isomorphism. 
Since  ${\rm HK}^i({\bb A}^n_C)=0$, from diagram \eqref{nov}, we obtain indeed that
${\rm Ker}({\rm HTlog}_{U,i}) =0$. 

  In the case of the torus $X={\mathbb G}^n_{m,C}$, we know from section \S\ref{nie2} that the map $H^i_{\proeet}({\bb G}^n_{m,C}, \Z_p)\to {\rm HK}^i({\bb G}^n_{m,C})$ is injective. This implies that  the map $f_2$ in diagram \eqref{nov} is surjective and so is the map $f_3$. It follows that
  $$
  {\rm Ker}({\rm HTlog}_{U,i}({\mathbb G}^n_{m,C})) \simeq {\rm HK}^i({\bb G}^n_{m,C})/H^i_{\proeet}({\bb G}^n_{m,C}, \Z_p)\simeq \wedge^i(\Q_p/\Z_p)^n,
  $$ as wanted.  

  Finally for the Drinfeld space $X={\mathbb H}^n_{C}$, the argument is analogous to the one in the case of the torus (using Proposition~\ref{drinfeld1}) and we get
  $$
    {\rm Ker}({\rm HTlog}_{U,i}({\mathbb H}^n_{C})) \simeq {\rm HK}^i({\bb H}^n_{C})/H^i_{\proeet}({\bb H}^n_C, \Z_p)\simeq {\rm Sp}_i(\Q_p)^*/{\rm Sp}_i(\Z_p)^*\simeq {\rm Sp}_i(\Z_p)^{\vee},
  $$
  as claimed.
\end{proof}

\subsubsection{The affine space ${\mathbb A}^n_C$}
In the case of the affine space, the above  will allow us  to compute the group $H^i_{\eet}({\bb A}_C^n, {\bb G}_m)$ in any degree $i$. To do that, we start with  the following result, whose proof was suggested to us by Ben Heuer: 

\begin{proposition} \label{o-bar-aff}
Let $n$ be an integer. For any $i>0$, the cohomology groups 
$$H_v^i({\bb A}_C^n,\overline{\so}^\times),\quad H_{\proeet}^i({\bb A}_C^n,\overline{\so}^\times),\quad H_{\eet}^i({\bb A}_C^n,\overline{\so}^\times)
$$ are trivial.\end{proposition} 

\begin{proof} 
The case $i=1$ was already mentioned earlier and is treated in~\cite[Lem.\,6.7]{Heu21}. For $i >1$, the result follows immediately from  Lemma \ref{Ben1} below.
\end{proof}
\begin{lemma} \label{Ben1}For $n>0$, let $B^n$ denote the closed unit disc of dimension $n$ over $C$. The cohomology groups 
$$H_{v}^i(B^n, \overline{\so}^\times),\quad H_{\proeet}^i(B^n, \overline{\so}^\times),\quad H_{\eet}^i(B^n, \overline{\so}^\times)
$$ are trivial for  $i>0$. \end{lemma}

\begin{proof}

For $i=1$, the result was proved in~\cite[Lem.\,{6.6}]{Heu21}. Assume $i \ge 2$.  Using~\eqref{even2}, it is enough to prove that $H_{\eet}^i(B^n, \overline{\so}^\times)=0$. We proceed by induction on the dimension $n$ of the closed disc. If $n=1$, we know that $H^j_{\eet}(B^1, \so^{\times})$ is trivial for $j\geq 1$ (see \cite[Lemma 6.1.2]{Ber}). Since $B^1$ is affinoid we also have $H^j_{\eet}(B^1, \so)=0$, for $j\geq 1$,  hence from the exponential exact sequence~\eqref{even1}, we obtain that $H^i_{\eet}(B^1, \overline{\so}^{\times})$ is trivial for $i\geq 2$.

Assume that the result holds for an integer $n$. Consider the projection $\pi : B^{n+1} \to B^n$. We have $\pi_{*} \overline{\so}^\times =  \overline{\so}^\times$. Using the induction hypothesis, it suffices to show that $\FS:=\R^j\pi_*\overline{\so}^\times$ is trivial for $j>0$. 
 Let $x:=\Spa(L, L^+)$ be a geometric point of $B^{n}$. It is enough to show that the stalk $\FS_x$ of $\FS$ at $x$ is trivial.  We claim  that we have: 
\[ \FS_x \simeq  H^j_{\eet}(B^1 \times \Spa(L,L^+),\overline{\so}^\times). \]
Arguing as in the proof of \cite[Lem.\,4.10]{BHP}, we can pass, via a \v{C}ech argument, to an affinoid perfectoid space $Y$ over $C$  and reduce to proving that if $Y\sim \lim_sY_s$, where $\{Y_s\}$
 is a cofiltered inverse system of smooth rigid spaces over $C$,  then
$H^j_{\eet}(Y,\overline{\so}^\times)\simeq \colim_sH^j_{\eet}(Y_s,\overline{\so}^\times)$. But this was proved in \cite[Prop.\,3.2]{BHD}.

 Now,  the cohomology $H^j_{\eet}(B^1 \times \Spa(L,L^+),\overline{\so}^\times)$ is isomorphic to $H^j_{\eet}(B^1 \times \Spa(L,L^{\circ}),\overline{\so}^\times)$, which is trivial  for $j>0$ by {induction with base case given in}  \cite[Lem.\,{6.6}]{Heu21}. This  concludes the proof of the lemma.  
\end{proof}

\begin{corollary}\label{may1}
Let $n$ be an integer. We have the isomorphisms
\begin{enumerate}
\item $\nu^*_0:  H^{0}_{\eet}({\bb A}_C^n,{\bb G}_m)\stackrel{\sim}{\to} H^{0}_v({\bb A}_C^n,{\bb G}_m)$.
\item 
$
H_v^i({\bb A}_C^n,  {\bb G}_m) \simeq \Omega^{i}({\bb A}_C^n)^{d=0}(-i), \quad i\geq 1.
$
\item 
$$
H_{\eet}^i({\bb A}_C^n,  {\bb G}_m)\simeq \begin{cases}
0 & \mbox{ if  }  \,i=1, i\geq n+1;\\
 \Omega^{i}({\bb A}_C^n)^{d=0}(-i+1) & \mbox{ if  } \,2\leq i\leq n.
\end{cases}
$$
\end{enumerate}
Moreover, the   map $\nu^*_i:  H^{i}_{\eet}({\bb A}_C^n,{\bb G}_m){\to} H^{i}_v({\bb A}_C^n,{\bb G}_m)$ is  zero for $i\geq 1$. 
\end{corollary}

\begin{proof} We note that Proposition \ref{o-bar-aff}  implies that the canonical maps
$$
H_{\tau}^i({\bb A}_C^n,  U)\to H_{\tau}^i({\bb A}_C^n,  \so^{\times}),\quad \tau\in\{\eet,\proeet, v\},
$$
are isomorphisms. Hence the $v$-cohomology case follows from  Proposition \ref{Bexamples} and Corollary \ref{morn12}.

For the \'etale cohomology, following Proposition \ref{ses-leray}, we consider two cases. 
If $i\geq n+1$, then we have $\nu^*_i: H^i_{\eet}({\bb A}_C^n,\so^{\times})\stackrel{\sim}{\to} H^i_v({\bb A}_C^n,\so^{\times})=0$, as wanted. If $n\geq i\geq 2$, then we have the exact sequence \eqref{cold2}, which together with Proposition \ref{Bexamples} yields our claim. By Proposition \ref{ses-leray}), 
the map $\nu_i^*: H^{n+1}_{\eet}({\bb A}_C^n,\so^{\times}){\to} H^{n+1}_v({\bb A}_C^n,\so^{\times})$ is the zero map.
\end{proof}
\begin{remark} (1) We could replace $v$-topology   by pro-\'etale topology  in the computations in the statement of Corollary \ref{may1}. This is because the canonical map 
$H_{\proeet}^i({\bb A}_C^n,  {\bb G}_m)\stackrel{\sim}{\to}H_{v}^i({\bb A}_C^n,  {\bb G}_m)$, $i\geq 0$,  is an isomorphism (by the spectral sequence \eqref{leray_nu} and its pro-\'etale analogue, Lemma \ref{log}, and Proposition \ref{omega}). 

 (2) For comparison, we note that over the complex numbers, the analogue of the \'etale part of Corollary \ref{may1} fails: the relevant groups are all trivial (use the exponential sequence and the fact that the affine space is Stein).
\end{remark}

\end{document}